\documentclass[
]{amsart}
\usepackage[foot]{amsaddr}

\long\def \forcecommand #1{\providecommand{#1}{}\renewcommand{#1}}

\newcommand*{\Root}{.}

\usepackage{newclude}
\usepackage[utf8]{inputenc}
\usepackage[english]{babel}
\usepackage[T1]{fontenc}
\usepackage{float}

\usepackage{graphics}
\usepackage{url}
\usepackage[round, comma, authoryear]{natbib}
\usepackage{natbib}
\usepackage{hyperref}

\usepackage[parfill]{parskip}
\usepackage[margin=1.5in]{geometry}
\usepackage{booktabs}
\usepackage{stackengine}
\usepackage{enumitem}
\usepackage{subcaption}
\usepackage[font=footnotesize]{caption}

\usepackage{amsmath,amsfonts,amsthm,amssymb,bbm}
\usepackage{bm}
\usepackage{listings}
\usepackage{accents}
\usepackage{euscript}
\usepackage{mathtools}
\usepackage{physics}
\usepackage{xargs}
\usepackage{cases}
\usepackage{enumitem}
\usepackage{mathrsfs}

\hypersetup{breaklinks=false,
colorlinks=true,
citecolor=black!50,
urlcolor=black!50,
linkcolor=black!50
}

\usepackage{tikz}
\definecolor{dodgerblue}{rgb}{0.12, 0.56, 1.0}

\newtheoremstyle{define}{10pt}{0pt}{\itshape}{}{\bf}{.}{.5em}{}
\newtheoremstyle{exmp}{10pt}{0pt}{\itshape}{}{\bf}{.}{.5em}{}
\newtheoremstyle{rmrk}{10pt}{10pt}{}{}{\bf}{.}{.5em}{}
\newtheorem{theorem}{Theorem}[section]
\newtheorem{proposition}[theorem]{Proposition}
\newtheorem{corollary}[theorem]{Corollary}
\newtheorem{lemma}[theorem]{Lemma}
\newtheorem{remark}[theorem]{Remark}
\newtheorem{definition}[theorem]{Definition}
\theoremstyle{define}
\newtheorem{example}{Example}[section]
\AtEndEnvironment{example}{\null\nobreak\hfill$\diamondsuit$}%

\usepackage[frak=euler,scr=boondoxo,bb=pazo]{mathalpha}
\DeclareFontFamily{U}{BOONDOX-calo}{\skewchar\font=45 }
\DeclareFontShape{U}{BOONDOX-calo}{m}{n}{
  <-> s*[1.05] BOONDOX-r-calo}{}
\DeclareFontShape{U}{BOONDOX-calo}{b}{n}{
  <-> s*[1.05] BOONDOX-b-calo}{}
\DeclareMathAlphabet{\calo}{U}{BOONDOX-calo}{m}{n}
\SetMathAlphabet{\calo}{bold}{U}{BOONDOX-calo}{b}{n}
\DeclareMathAlphabet{\calo}{U}{BOONDOX-calo}{b}{n}

\usepackage[ruled,vlined]{algorithm2e}
\SetKwInOut{Input}{Input}
\SetKwInOut{Output}{Output}

\usepackage{tikz}
\usetikzlibrary{cd}
\definecolor{dodgerblue}{rgb}{0.12, 0.56, 1.0}

\usepackage{amsfonts}
\usepackage{\Root/inputs/macros}

\renewcommand{\thesubfigure}{\Alph{subfigure}}
\makeatletter
\renewcommand\p@subfigure{\thefigure\,({\thesubfigure})\@gobble}
\makeatother

\usepackage{mathrsfs}
\usepackage[]{lipsum}

\usepackage{xr}
\usepackage{xr-hyper}
\usepackage{hyperref}
\usepackage[capitalise]{cleveref}

\usepackage[normalem]{ulem}

\title
[On the limits of topological data analysis for statistical inference]
{On the Limits of Topological Data Analysis\\ 
for Statistical Inference}

\author[Vishwanath]{Siddharth Vishwanath$^{1,*}$}\thanks{$^*$ Corresponding Author}
\email[Corresponding author]{svishwanath@ucsd.edu}
\address{$^1$University of California San Diego, 9500 Gilman Dr, La Jolla, CA 92093}

\author[Fukumizu]{Kenji Fukumizu$^2$}
\email{fukumizu@ism.ac.jp}

\author[Kuriki]{\\ Satoshi Kuriki$^2$}
\email{kuriki@ism.ac.jp}
\address{$^2$The Institute of Statistical Mathematics, 10-3 Midoricho, Tachikawa, Tokyo 190-8562, Japan}

\author[Sriperumbudur]{Bharath K. Sriperumbudur$^3$}
\email{bks18@psu.edu}
\address{$^3$The Pennsylvania State University, University Park, PA 16802, USA}

\begin{document}
\begingroup
\maketitle

\begingroup
\begin{abstract}
    Topological data analysis has emerged as a powerful tool for extracting the metric, geometric and topological features underlying the data as a multi-resolution summary statistic, and has found applications in several areas where data arises from complex sources. In this paper, we examine the use of topological summary statistics through the lens of statistical inference. We investigate necessary and sufficient conditions under which \textit{valid statistical inference} is possible using {topological summary statistics}. Additionally, we provide examples of models that demonstrate invariance with respect to topological summaries.
\end{abstract}%
\section{Introduction}\label{sec:intro}

Let $\Xn = \pb{\Xv_1, \Xv_2, \dots, \Xv_n}$ be a collection of points observed \iid{} at random from a probability distribution $\pr$ on {$\X \subseteq \R^d$}. The core objective of the statistical inference is to employ a test statistic $\T(\Xn)$ to infer meaningful information underlying the data-generating mechanism $\pr$. A good choice of the statistic $\T(\cdot)$ sheds light on some population quantity of interest, $\psi(\pr)$, which encodes the essential information underlying the sample $\Xn$. For instance, choosing $\T(\Xn) = (X_1 + X_2 + \dots + X_n)/n$ to be the sample mean sheds light on the population mean $\psi(\pr) = \int_{\X} \xv \cdot d\pr(\xv)$. In particular, a good choice of $\T(\cdot)$ for $\psi(\pr)$ enables a practitioner to test hypotheses $H_0: \pr \in \mathcal{P}_0$ vs. $H_1: \pr \in \mathcal{P}_1$ related to the data generating mechanism and to perform valid statistical inference.

Recent years have witnessed the availability of high-dimensional data from unconventional sources such as text and images, and it has become increasingly important to employ statistical summaries that encode the subtle features that underlie the data. To this end, topological data analysis (TDA) has emerged as an important tool for {extracting geometric and topological features underlying the data}. Summary statistics obtained using TDA are particularly attractive since they are multi-resolution summaries that encode the metric, geometric \textit{and} topological features underlying the data. Therefore, topological summary statistics, e.g., Betti numbers \citep{chung2019statistical}, Euler characteristic curves \citep{meng2022randomness}, persistence functions \citep{sorensen2020revealing} and persistence diagrams \citep{bendich2016persistent}, have been employed in statistical tests for data arising from complex sources such as neuroscience \citep{centeno2022hands}, cosmology \citep{adler2017modeling} and proteomics \citep{gameiro2015topological}. Despite the widespread adoption of TDA in data-analysis pipelines, a formal framework for statistical inference using these statistics is still limited. 

More precisely, topological summary statistics, represented as $\T: \X^n \rightarrow \mathcal{S}$, are measurable mappings (w.r.t. $\pr^{\otimes n}$) such that the data $\Xn$ is mapped to {an} element $\T(\Xn) \in \mathcal{S}$ in a suitable topological summary space. However, several topological summaries are challenging to analyze in a formal statistical context owing to the unusual mathematical structure of the summary space $\mathcal{S}$. For example, \cite{mileyko2011probability} show that the summary space for persistence diagrams is an Alexandrov space with non-negative curvature bounded from below. Working with topological representations on such summary spaces $\mathcal{S}$ are not amenable for employing classical tools of statistical inference owing to the bounded curvature of geodesics, non-uniqueness of Fr\'{e}chet means, and absence of any Hilbertian structure \citep{divol2019understanding,turner2014frechet}.

In contrast, well-behaved topological summaries\footnote{e.g., when $\mathcal S = \R^k$, or, more generally, when $\mathcal S$ is a Hilbert space or a Banach space.}---such as Betti numbers and persistent Betti numbers---have been studied extensively in a probabilistic context. Based on the central ideas from \cite{penrose2003random}, the large sample behavior of $\T(\Xn)$ branches into three qualitatively different regimes: the \textit{sparse}, \textit{thermodynamic} and \textit{dense} regimes which depend on the resolution at which the observations are examined relative to the number of samples $n \rightarrow \infty$. In particular, recent advances in the discipline have established the existence of limiting quantities $\mu(\pr)$ and $\Sigma(\pr)$ such that
\eq{
    \frac{\T(\Xn)}{n} \stackrel{p}{\rightarrow} \mu(\pr), %
    \quad {and} \quad %
    \frac{\T(\Xn) - \mu(\pr)}{\sqrt{n}} \stackrel{d}{\rightsquigarrow} \mathcal{N}\qty\big( \zerov, \Sigma(\pr) ),\nn
}
as $n \rightarrow \infty$ under a suitable asymptotic regime. For example, the law of large numbers and central limit theorem for the \textit{Betti numbers} of random geometric complexes is established in \cite{kahle2011random,kahle2013limit,Yogeshwaran_2015,yogeshwaran2017random,goel2019strong,trinh2017remark,bobrowski2015topology,bobrowski2014distance}. Similar results for the \textit{persistent Betti numbers} of random geometric complexes are established in \cite{duy2016limit,krebs2019asymptotic,can2020random}. In \cite{thomas2019functional,owada2020limit,krebs2021approximation} the topological summaries are treated as stochastic processes, and analogous convergence results are studied in the Skorohod metric.  These results establish that studying fundamental topological quantities in a random setting guarantee stability in a probabilistic sense, and pave the way to more detailed statistical investigation. In particular, \cite{fasy2014confidence,biscio2020testing,roycraft2020bootstrapping,dlotko2022topology} have provided methods for incorporating topological summaries in a formal statistical inference setting. Notwithstanding, it is tempting to ask:

\quad\emph{When are these topological summaries $\T(\Xn)$ meaningful for statistical inference?}

The \textit{central objective} of this work is to investigate a variant of this question: \textit{when are the topological summaries insufficient for statistical inference?} To this end, we characterize conditions under which the limiting distribution of $\T(\Xn)$ fails to be injective. 

\subsection{Contributions}\label{contributions}

In the deterministic setting, even without imposing the probabilistic structure, the injectivity of metric spaces via topological transforms is largely an open problem \cite{oudot2018inverse}. Therefore, we make a foray into this question using a slightly simplified approach. Given $\Xn$ sampled \iid{} from $\pr$ where $\supp(\pr) = \X \subset \R^D$ and $\dim(\X) = d \le D$, and for $0 \le s, t \le \infty$ the topological summary we investigate are the \textit{persistent Betti numbers}, $\beta^{s, t}_k\qty\big(\K\qty\big({n^{1/d}\Xn}))$ for each $0 \le k \le d$, i.e.,
\eq{
    \T(\Xn) = \qty\Big( \beta^{s, t}_0\qty\big(\K({n^{1/d}\Xn})), \beta^{s, t}_1\qty\big(\K({n^{1/d}\Xn})), \dots, \beta^{s, t}_d\qty\big(\K({n^{1/d}\Xn})) ) {\in \Z_+^d},\nn
}
where, as outlined in Section~\ref{sec:background}, $\K(n^{1/d}\Xn)$ corresponds to the \v{C}ech complex constructed using $\Xn$ in the thermodynamic regime. Examining the behavior of the Betti numbers collectively serves as a stepping-stone to understanding the behavior of more complex topological invariants in the context of persistent homology.

Although the exact {sampling} distribution for $\T(\Xn)$ is difficult to characterize, we can investigate its usefulness in the following asymptotic sense. If $\T(\pr)$ denotes the limiting distribution of $n^{-1/2} \T(\Xn)$ {i.e., $$n^{-1/2} \T(\Xn) \stackrel{d}{\rightsquigarrow} \T(\pr),$$} in this paper we investigate conditions under which the map $\pr \mapsto \T(\pr)$ fails to be injective, i.e., for two different probability distributions $\pr\neq\qr$, we have that $\T(\pr)=\T(\qr)$. 

At this point, it becomes instructive to consider a class of distributions $\P = \pb{\pr_\theta : \theta \in \Theta}$ indexed by a parameter set $\Theta$. A {statistic} $\T\pa{\Xn(\theta)}$ is said to be {sufficient} for the model $\P$ when the probabilistic and statistical information underlying the observations $\Xn(\theta) \sim \pr_\theta$ is faithfully encoded in the statistic $\T(\Xn(\theta))$, and plays a central role in statistical inference\footnote{{Specifically, $\T(\Xn)$ is sufficient for the model $\P$ if the joint distribution of $\Xn(\theta)$ admits the factorization $f_{\Xn}(\xv_{1:n}; \theta) = h(\xv) \cdot g(\theta, \T(\xv_{1:n}))$}}. The injectivity of $\T$ is closely related to the notion of \textit{sufficiency} of $\T$. In particular, when injectivity fails, the limiting distribution {of} $\T(\Xn(\theta))$ ultimately provides no information about the parameter $\theta$ underlying $\Xn(\theta)$. This qualitative behavior, which is complementary to the notion of sufficiency, is called \textit{ancillarity}; our results can be viewed as characterizing conditions for $\P$ under which the topological summary statistic $\T(\Xn(\theta))$ is \textit{asymptotically ancillary} for the model $\mathcal{P}$.\smallskip

Our main contributions are the following:

\begin{enumerate}[label=(\Roman*)]
\item We introduce the notion of \betaequivalence{} which \mbox{characterizes} distributions for which the mapping $\T$ fails to be injective. We {examine} conditions under which a parametric class of distributions $\P = \pb{\pr_\theta: \theta \in \Theta}$ admits \betaequivalence{}, by introducing an alternate equivalence relationship, called \Fequivalence{}, whereby distributions satisfying \Fequivalence{} are also guaranteed to satisfy \betaequivalence{} (Lemma~\ref{lem:inclusion}). As a consequence, when $\pr$ and $\qr$ admit \Fequivalence{}, any statistical test for distinguishing between $\Xn \sim \pr$ and $\Ym \sim \qr$ using $\T(\Xn), \T(\Ym)$ has vanishing power as {$n, m \rightarrow \infty$}.

\item By imposing an algebraic structure on $\P$, and using the notion of \emph{group maximal invariance} \citep{wijsman1990invariant, eaton1989group}, in Theorem \ref{thm:group_invariance}, we provide the necessary and sufficient conditions for the family $\mathcal{P}$ to admit \Fequivalence{}. We illustrate this result through some supporting examples in Section~\ref{sec:invariance2}.

\item Next, in Theorem~\ref{thm:nonlinear_invariance} we relax the algebraic assumptions on $\P$ and investigate sufficient conditions for \Fequivalence{} when the underlying space $\X$ admits a smooth fiber bundle structure. In contrast to (II), our method here is more constructive and is illustrated through several examples in Section~\ref{sec:invariance3}.

\item Lastly, in Theorem \ref{thm:orthogonal}, we present {a necessary} and sufficient condition for \Fequivalence{} to hold based on the geometry of the score function $\nabla_\theta \log f_\theta(\xv)$. In contrast to (II) and (III), this result provides a simple method to verify if a given class of distributions, $\P$, admits \Fequivalence{}.
\end{enumerate}

The rest of the paper is organized as follows. In Section \ref{sec:background}, we provide a background on the probabilistic, topological and statistical tools needed. Sections~\ref{sec:injectivity},~\ref{sec:invariance2}~and~\ref{sec:invariance3} contain the main results, and missing proofs are collected to Section~\ref{sec:proofs}. In Section \ref{sec:discussion} we discuss possible extensions and collect the supplementary results in the Appendix.
\section{Background}\label{sec:background}

In this section, we provide background on probabilistic, topological, and statistical tools needed for the rest of the paper. We use $\Xv, \Yv, \Zv$ to denote random variables from a probability space $\pa{\Omega,\mathcal F,\pi}$ taking values in a measurable space $\pa{\mathcal{X},\mathcal{B}\pa{\mathcal{X}}}$, and we use $\Xv \perp \Yv$ to indicate that $\Xv$ and $\Yv$ are independent. We assume that the space on which the random values are observed, $\X$ is sufficiently regular, i.e., either $\X \subset \R^d$ contains an open subset of $\R^d$, or $\X$ is a compact \mbox{$\mathcal{C}^1$-manifold} of dimension $d<D$.  For all probability distributions $\pr$, defined on $\X$, we assume the existence of a Radon-Nikodym derivative $f = {d\pr}/{d\mu}$ w.r.t.~a canonical choice of the base measure on $\mu$ on $\X$, e.g., the $d$-dimensional Lebesgue measure, denoted $\lambda_d$, when $\X$ contains an open subset of $\R^d$. We further assume that $f \in L^p\pa{\X, \mu}$ for all $1 \le p < \infty$. The Jacobian of a differentiable function $f$ is denoted by $\mathbf{D} f$ and $\norm{\mathbf D f} \defeq \abs{ \det\qty(\mathbf{D}f)}$ denotes the absolute value of its determinant.


\subsection{Betti Numbers and Persistent Homology}
\label{sec:ph}

Given a collection of observations $\mathbb{X}_n = \pb{\xv_1,\xv_2, \dots, \xv_n}$ in a metric space $\pa{\mathcal{X},d}$, and for a given spatial resolution $r>0$, the topology underlying the points at resolution $r$ is encoded in a simplicial complex $\mathcal{K}\pa{\mathbb{X}_n,r} \subseteq 2^{\Xn}$. The simplicial complex can be constructed in several ways (see, for example, Edelsbrunner \& Harer \cite{edelsbrunner2010computational}). In particular, the \text{\v{C}ech complex} is given by 
\eq{
{\mathcal{K}\pa{\mathbb{X}_n,r} \defeq \pb{\sigma \subseteq \Xn : \bigcap_{\xv \in \sigma}B\pa{\xv, r} \neq \varnothing}},
\label{eq:cech}
}

where {$B(\xv, r) = \pb{\yv \in \R^D: \norm{\xv - \yv} < r}$}. For $0 \le k \le d$, the $k^{th}$-\textit{homology} of a simplicial complex $\mathcal{K}$, given by $H_k\pa{\mathcal{K}}$ is an algebraic object encoding its topology as a vector-space {(over a fixed field $\mathsf{k}$, typically taken to be $\mathbb Z_2$; \citealp{hatcher2005algebraic})}. Using the Nerve lemma \citep{borsuk1948imbedding}, $H_k\pa{\mathcal{K}\pa{\mathbb{X}_n,r}}$ is isomorphic to the homology of its union of $r$-balls, $H_k\pa{\bigcup_{i=1}^{n}B\pa{\xv_i, r}}$. The $k^{th}$-\textit{Betti number} is defined as
\eq{
\beta_k\pa{\mathcal{K}\pa{\mathbb{X}_n,r}} \defeq \dim\pa{H_k\pa{\mathcal{K}\pa{\mathbb{X}_n,r}}}.\nonumber
}
It counts the number of $k$-dimensional voids or {\textit{non-trivial $k$--cycles} in $\mathcal{K}\pa{\mathbb{X}_n,r}$}. The ordered sequence $\pb{\mathcal{K}\pa{\mathbb{X}_n,r}}_{r > 0}$ forms a \textit{filtration}, encoding the evolution of topological features over a spectrum of resolutions. For $0<s<t$, $\bigcup_{i=1}^{n}B\pa{\xv_i, s} \subset \bigcup_{i=1}^{n}B\pa{\xv_i, t}$ and the simplicial complex $\mathcal{K}\pa{\mathbb{X}_n, s}$ is a \textit{sub-simplicial complex} of $\mathcal{K}\pa{\mathbb{X}_n, t}$. Their homology groups are associated with the induced linear map 
$$\iota_s^t:H_k\pa{\mathcal{K}\pa{\mathbb{X}_n, s}} \hookrightarrow H_k\pa{\mathcal{K}\pa{\mathbb{X}_n, t}}$$ 
and the $k^{th}$ order \textit{$(r,s)$-persistent Betti number}, given by 
$$
\beta_k\pa{\K(\mathbb{X}_n, s, t)} = \textup{rank}\pa{\iota_s^t},
$$
counts the number of non-trivial cycles which are born at or before $s$ and have a death after $t$. The $k^{th}$-{persistence diagram}, denoted by $\dgm_k\pa{\mathbb{X}_n}$, {is defined as the collection} $\{(b_i,d_i)\}_i$ of birth-death pairs associated with the non-trivial cycles from the filtration. We refer the reader to \cite{hatcher2005algebraic,edelsbrunner2010computational} for a comprehensive introduction.

\subsection{Asymptotic Regimes}
\label{sec:prob}

A point process $\Phi$ is a locally finite, random counting measure defined on the measurable space $\pa{\Omega,\mathcal F}$ whose random elements take values in $\mathcal X$, and for a $B \in \mathcal{B}\pa{\mathcal{X}}$, the random variable $\Phi\pa{B}$ measures the number of elements of $\Phi$ in $B$. $\Xn$ {is called a Binomial point process} if $\Xn = \pb{\Xv_1,\Xv_2,\dots,\Xv_n}$ is \iid{} with distribution $\pr$. The Binomial point process can equivalently be represented by the random measure $\Phi_n = \sum_{i=1}^{n}{\delta_{\Xv_i}}$, from which it follows that $\Phi_n(B) = \textup{Bin}(n, \pr(B))$ for all $B \in \mathcal{B}\pa{\mathcal{X}}$.

In Section \ref{sec:ph}, the collection of points $\pb{\xv_1, \xv_2, \dots, \xv_n}$ were fixed points from a space $\pa{\mathcal{X},d}$. The probabilistic setting deals with a random collection of points ${\Xn = \pb{\Xv_1,\Xv_2, \dots, \Xv_n}}$. The analysis of the asymptotic behavior of the topological quantities depends on how the radii of the balls for the \v{C}ech complex, $r_n$, decays relative to $n$. If the radii $r_n$ decay too quickly, then the associated simplicial complexes fail to recognize the higher dimensional simplices, resulting in sparsely connected points. On the flip side, if the radii $r_n$ decay too slowly, then all the points become densely connected. At a critical rate of decay for $r_n$ one can observe a phase transition. {This is illustrated in Figure \ref{fig:regime}.}

\begin{figure}
  \centering
  \begin{subfigure}{0.32\textwidth}
    \centering
    \includegraphics[width=\textwidth]{\Root/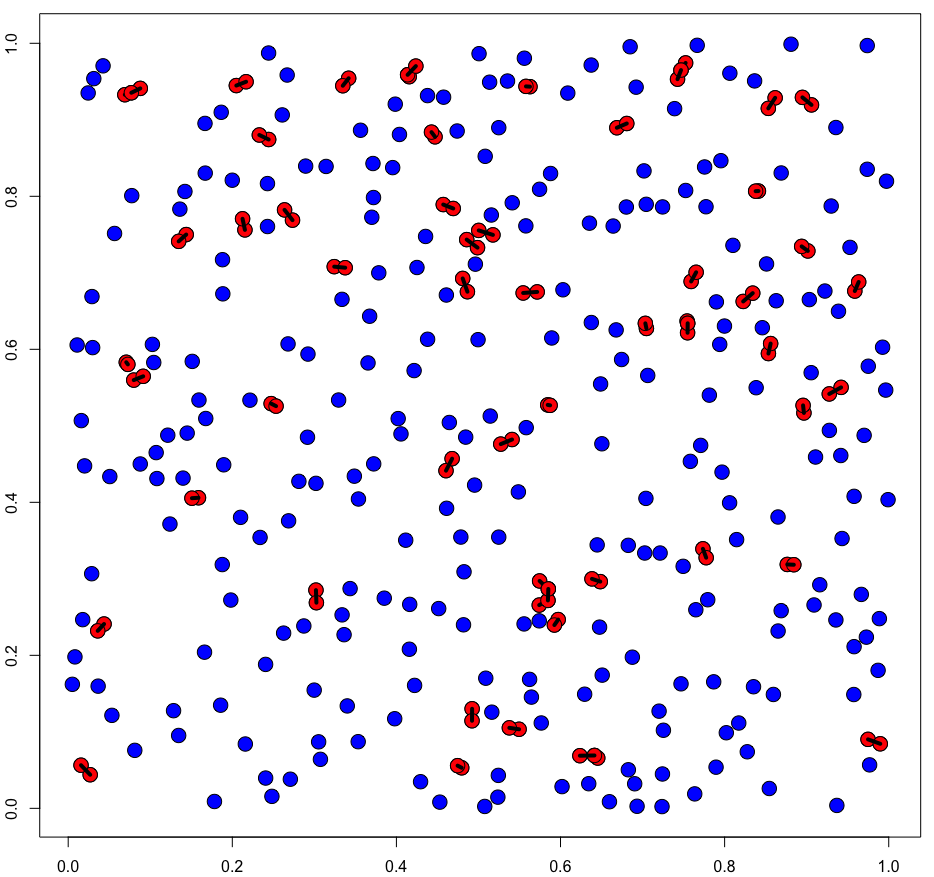}
    \caption{Sparse regime}
  \end{subfigure}
  \begin{subfigure}{0.32\textwidth}
    \centering
    \includegraphics[width=\textwidth]{\Root/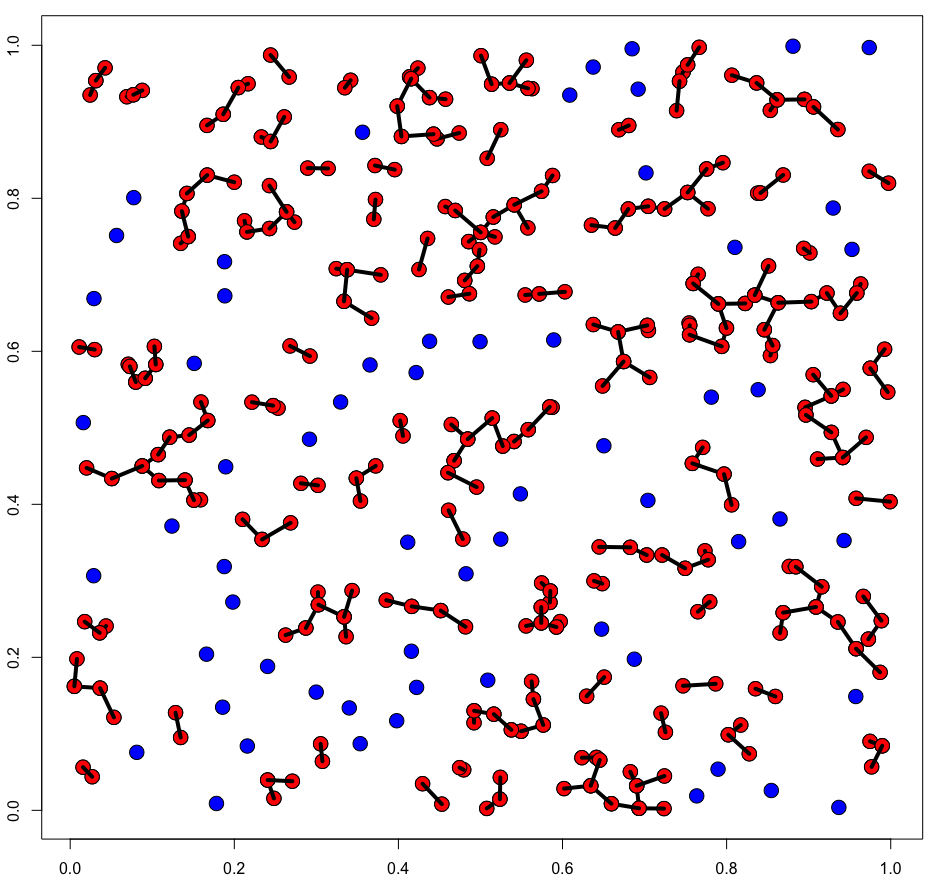}
    \caption{Thermodynamic regime}
  \end{subfigure}
  \begin{subfigure}{0.32\textwidth}
    \centering
    \includegraphics[width=\textwidth]{\Root/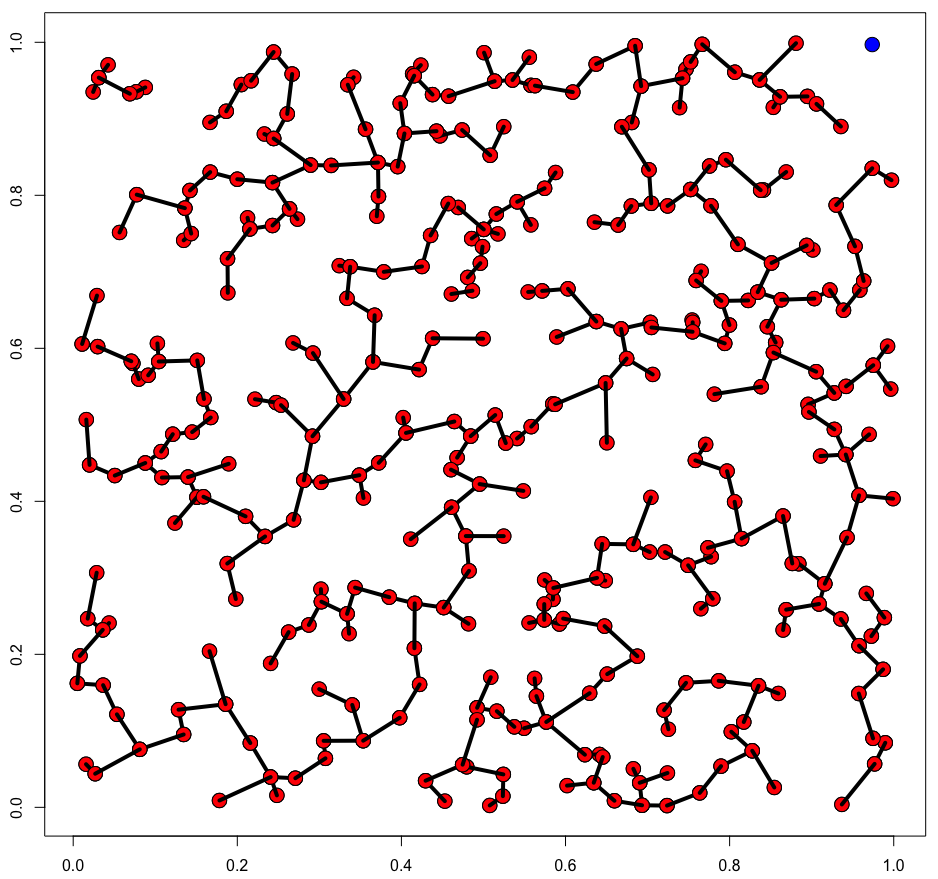}
    \caption{Dense regime}
  \end{subfigure}    
  \caption{Illustration of different asymptotic regimes}
  \label{fig:regime}
\end{figure}

Formally, the \textit{thermodynamic} regime corresponds to the case when the expected number of points inside a ball of radius $r_n$ has constant order, i.e., {if $\pr \ll \lambda_d$ then}
\eq{
\E\big(\Phi(B(\xv, r_n)\big)\big) = n \pr( \Xv \in B(\xv, r_n) ) \propto n r_n^d {\longrightarrow} c \qq{as} {n \rightarrow \infty}, \nonumber
}
whenever $r_n = c {n^{-{1}/{d}}}$. Faster rates of decay, $r_n = o({n^{-{1}/{d}}})$, corresponds to the \textit{sparse} regime, and slower rates of decay, $r_n = \omega({n^{-{1}/{d}}})$, corresponds to the \textit{dense} regime. 

From the definition of the \v{C}ech complex in \eref{eq:cech}, it is easy to see that in the thermodynamic regime $r_n = t n^{-1/d}$, $\K(\Xn, r_n)$ is equivalently obtained by holding the radius fixed at $t$ and rescaling the original points to get $\K(n^{1/d}\Xn, t)$. Specifically, taking $r_n = t n^{-1/d}$ if $\sigma \subset \Xn$ is a $k-$simplex of $\K(\Xn, r_n)$ then for $n^{1/d}\sigma \subset n^{1/d}\Xn$ it follows that
\eq{
  \bigcap_{\yv \in n^{1/d}\sigma}B\pa{\yv, t} \neq \varnothing \quad \Leftrightarrow\quad \bigcap_{\xv \in \sigma}B\pa{\xv, r_n} \neq \varnothing,\nn
}
and, therefore, $\K(\Xn, r_n) \simeq \K(n^{1/d}\Xn, t)$ {is a simplicial isomorphism}. Similarly, when $r_n, R_n$ are such that $r_n = s n^{-1/d}$ and $R_n = t n^{1/d}$ in the thermodynamic regime, the $(r_n, R_n)$-persistent Betti numbers are equivalently given by 
\eq{
  \beta_k^{s, t}\qty( \K(n^{1/d}\Xn) ) \defeq \beta_k\qty( \K(n^{1/d}\Xn, s, t) ) \equiv \beta_k\qty( \K(\Xn, r_n, R_n) ).\nn
}

\section{Injectivity of $\T(\pr)$, \betaequivalence{} and \Fequivalence{}}
\label{sec:injectivity}

We begin by describing the setting in which we examine injectivity. Let $\Xn = \pb{\Xv_1,\Xv_2, \dots, \Xv_n}$ be sampled \iid{} from a probability distribution $\pr$ with density $f$ on $\X \subseteq \R^D$. The topological summary we are interested in {is the collection of persistent Betti numbers} for the random \v{C}ech complexes in the thermodynamic regime,
\eq{
\T\pa{\Xn} \defeq \Big(\beta_0^{s, t}( \K(n^{1/d}\Xn) ), \beta_1^{s, t}( \K(n^{1/d}\Xn) ), \dots , \beta_d^{s, t}( \K(n^{1/d}\Xn) )\Big),\nonumber
}
for $0 \le s < t$. The thermodynamic limit for each $\beta_k\pa{\mathcal{K}\pa{\Xn,r_n}}$, $0 \le k \le d$ has been established in \cite{kahle2013limit,yogeshwaran2017random,trinh2017remark,goel2019strong,krebs2019asymptotic,can2020random}. In particular, we rely on the following characterization of the thermodynamic limit found in \citep[Proposition~3.1~\&~Theorem~4.5]{krebs2019asymptotic} which follows from a slight restatement of the conditions in \cite[Theorem~1.1]{goel2019strong}. 

\begin{proposition}[{\citealp[Theorem~1.1]{goel2019strong}; \citealp[Proposition~3.1~\&~Theorem~4.5]{krebs2019asymptotic}}]
  {Let $\Xn \subset \X$ sampled \iid{} from $\pr$ with density $f$, such that $f \in L^p\pa{\mu, \X}$ for all $1 \le p < \infty$. {For fixed $0 \le s < t \le \infty$}, there exist fixed functions $\gamma_k, \varsigma_k$ depending only on $k$ such that
  \eq{
  \lim_{n \rightarrow \infty}\f{\beta^{s, t}_k\qty(\K(n^{1/d}\Xn))}{n} \defeq \mu_k(\pr;s, t) = \int\limits_{\X}{\gamma_k\pa{f(\xv)^{1/d}s, f(\xv)^{1/d}t} f(\xv) d\xv},\nonumber
  }
  and
  \eq{
  \frac{\beta^{s, t}_k\qty(\K(n^{1/d}\Xn)) - \mu_k(\pr;s, t)}{\sqrt{n}} \stackrel{d}{\rightsquigarrow} \mathcal{N}(0, \sigma_k^2(s, t)),\nn
  }
  where $\sigma_k^2(s, t) \defeq \int_\X\varsigma_k(f(\xv)^{1/d}s, {f(\xv)^{1/d}t}) f(\xv) d\xv$.
  }
  \label{prop:stronglaw}
\end{proposition}

\begin{remark}
  While we focus on Betti numbers associated with random \cech{} complexes in the thermodynamic regime, as noted in \citet[Section 1.3]{Yogeshwaran_2015}, the results in this paper will extend to the Vietoris-Rips, alpha, witness and cubical complexes as well. For example, the probabilistic results in \citet[Theorem 1.1]{goel2019strong} and \citet[Theorem 3.3]{trinh2017remark} hinge on establishing moment bounds using the topological additivity property in \citet[Lemma 2.2]{goel2019strong} and \cite[Lemma 2.1]{trinh2017remark} respectively---which extend to the aforementioned simplicial complexes. In particular, \citet[Section~4.2]{roycraft2020bootstrapping} describes some general conditions an abstract simplicial complex needs to satisfy for the results to hold. 
\end{remark}

A key observation is to note that the limiting quantities in Proposition \ref{prop:stronglaw} can equivalently be written as the statistical functionals
\eq{
\mu_k(\pr;s, t) &\defeq \int_{\X}{\gamma_k\pa{f(\xv)^{1/d}s, f(\xv)^{1/d}t} f(\xv) d\xv},\nonumber\\[0pt]
\sigma^2_k(\pr; s, t) &\defeq \int_\X\varsigma_k(f(\xv)^{1/d}s, {f(\xv)^{1/d}t}) f(\xv) d\xv\nn,
}
where $f$ is the density associated with $\pr$. Therefore, the thermodynamic limit, ${\beta_k(\pr; s, t) \defeq \qty\big(\mu_k(\pr;s, t),\ \sigma^2_k(\pr; s, t))}$, encodes the limiting behavior of $n^{-1/2}\beta_k^{s, t}(\K(n^{1/d}\Xn))$.

With this background, we define the notion of $\beta_k$-equivalence under which two distributions $\pr, \qr$ admit the same thermodynamic limit $\beta_k(\pr; s, t)$ for all $0 \le s < t$. We are interested in families of distributions which admit $\beta_k$-equivalence for each $k$, i.e., {$\beta_k(\pr;s, t) = \beta_k(\qr;s, t)$} for each $k\ge 0$. We call such a family of distributions \betaequivalent{}.

\begin{definition}[$\beta$-equivalence]
  Two distributions $\pr, \qr$ are said to admit $\beta_k$-{equivalence} if $\beta_k(\pr;s, t) = \beta_k(\qr;s, t)$ for all $0 \le s < t$. Furthermore, $\pr$ and $\qr$ are said to admit \betaequivalence{}, denoted $\pr \beq \qr$, if $\pr$ and $\qr$ admit $\beta_k$-{equivalence} for all $k \ge 0$. Moreover, a family of distributions $\P$ admits \betaequivalence{} if $\pr \beq \qr$ for all $\pr, \qr \in \P$.
  
\end{definition}

Establishing \betaequivalence{} for distributions directly is infeasible because the exact forms of the quantities $\gamma_k$ and $\varsigma_k$, as described in Proposition~\ref{prop:stronglaw}, are typically unknown. To circumvent this challenge, we introduce an alternative equivalence, termed \Fequivalence{}.

\begin{definition}[\Fequivalence{}]\label{def:Fequivalence}
  Consider two probability distributions $\pr, \qr$ with probability density functions $f$ and $g$, respectively. Let $\Xv \sim f$ and $\Yv \sim g$ be two random variables with distribution $\pr$ and $\qr$, respectively. Then, $\pr$ and $\qr$ are said to be \Fequivalent{}, denoted $\pr \Feq \qr$ (equivalently, $f \Feq g$) if $f(\Xv) \distas g(\Yv)$. A family of distributions $\P$ is said to admit \Fequivalence{} if $\pr \Feq \qr$ for all $\pr, \qr \in \P$.
\end{definition}

The following result establishes the relationship between distributions that admit \Fequivalence{} and those that admit \betaequivalence{}. 

\begin{lemma}
  If $\P$ is a family of probability distributions that admit \Fequivalence{}, then $\P$ also admits \betaequivalence{}. 
  \label{lem:inclusion}
\end{lemma}

In other words, \Fequivalence{} is a sufficient condition for distributions to admit \betaequivalence{}, i.e., distributions satisfying \Fequivalence{} are also guaranteed to satisfy \betaequivalence{}.

\setlist[enumerate]{leftmargin=2em,itemindent=2em,parsep=5pt,itemsep=7pt}

\begin{remark}
  We make the following observations regarding distributions which admit \Fequivalence{}.
  \begin{enumerate}[leftmargin=0pt,label=\textup{(\roman*)}]
    \item Let $\P(\Theta) = \pb{\ptheta : \theta \in \Theta}$ be a parametric family of distributions. A statistic $\T(\Xn)$ is an \textit{ancillary statistic} for the model $\P(\Theta)$ if the distribution of $\T(\Xn)$ (w.r.t. $\pr_\theta^{\otimes n}$) does not depend on $\theta$. In a similar vein, $\T(\Xn)$ is \textit{approximately ancillary} for the model if the limiting distribution $\T(\pr_\theta)$ does not depend on $\theta$ \cite[Chapter~6.6]{severini2000likelihood}. This provides the following interpretation of \Fequivalence{}: if $\P(\Theta)$ admits \Fequivalence{}, then the topological summaries are approximately ancillary statistics for the model. 
    
    Given a fixed model $\P(\Theta)$, in general, there are no constructive techniques for determining ancillary statistics \citep{lloyd2004ancillary}. Our objective here is, however, somewhat complementary; given a fixed statistic $\T(\Xn)$, we investigate conditions under which the model $\P(\Theta)$ admits $T(\Xn)$ as an approximate ancillary statistic. 
    
    \item Notably, if we are given observations $\Xn(\theta) \sim \pr_\theta$ from a \betaequivalent{} family $\P(\theta)$, any level-$\alpha$ hypothesis test for $H_0: \theta \in \Theta_0$ vs. $H_1: \theta \in \Theta_1$ with $\Theta_0 \cap \Theta_1 = \varnothing$ using $\T(\Xn)$ as a test statistic will have vanishing power. This is a simple consequence of the fact that, in the limit, for any rejection region $B \in \scr{B}(\R^d)$,
    \eq{
    \pr_{_{\!\!H_1}}\qty\big(\T(\Xn) \in B)
    &\le \pr_{_{\!\!H_0}}\qty\big(\T(\Xn) \in B) + \abs\Big{\pr_{_{\!\!H_1}}\qty\big(\T(\Xn) \in B) - \pr_{_{\!\!H_0}}\qty\big(\T(\Xn) \in B )}\nn\\
    &\le \alpha + o(1),\nn
    }
    where the last inequality follows from the fact that the  $\alpha$ is the fixed type-I error rate and the second term follows from the fact that, due to \Fequivalence{}, as $n \rightarrow \infty$ the limiting distributions of $\T(\Xn)$ under the null and alternate hypothesis are the same, i.e., $\T(\pr_{\theta_0}) \distas \T(\pr_{\theta_1})$ for all $\theta_0 \in \Theta_0$ and $\theta_1 \in \Theta_1$.
    
    We, however, emphasize that the approximate ancillarity of topological summaries $\T(\Xn)$ for such models does not invalidate their usefulness in statistical inference; it only limits their ability to be used as the primary test statistic when the model admits \Fequivalence{}. In fact, ancillary statistics play a pivotal role in deriving efficient and optimal testing procedures in the framework of conditional inference \citep{efron1978assessing,ghosh2010ancillary}.
    
    \item Studying injectivity for Betti numbers collectively serves as a stepping-stone to understanding the behavior of several other topological invariants, e.g., the \textit{Euler characteristic} is an important topological invariant and is given by the alternating sum of Betti numbers,
    \eq{
    \chi\pa{\mathcal{K}\pa{\Xn,r}} \defeq \sum_{k=0}^{d}(-1)^k \beta_k\pa{\mathcal{K}\pa{\Xn, r}}.
    \label{eq:euler}
    }
    Unlike Betti numbers, the asymptotic behavior of the Euler characteristic of a random \cech{} complex exhibits interesting phenomena only in the thermodynamic regime \cite[Corollary 4.2]{bobrowski2014distance}, i.e., 
    $$
    \T_{\chi}(\Xn) = n^{-1/2}\chi\pa{\mathcal{K}({\Xn, tn^{-1/d}})}.
    $$ 
    By noting that $\beta_k(\K(\Xn, tn^{-1/d})) = \beta_k^{t, t}(\K(n^{1/d}\Xn))$ in \eref{eq:euler} and invoking the continuous mapping theorem, it is easy to see that families of distributions which admit \Fequivalence{} in the thermodynamic regime will also admit identical asymptotic behavior of the Euler characteristic. In particular, this shows that the goodness-of-fit test proposed in \cite{dlotko2022topology} will fail to distinguish between \betaequivalent{} point processes.
    
    \item Furthermore, note that the persistence diagram, which (informally) is given by
    $$
    \dgm_k(n^{1/d}\Xn)\! = \!\!\qty\Big{(b_i, d_i) : \exists\textup{$k$-dimensional feature with birth time $b_i$ \& death time $d_i > b_i$} },$$ 
    can be represented by a locally finite Radon measure (\citealp{duy2016limit,divol2019understanding})
    \eq{
    \Psi[\dgm(n^{1/d}\Xn)] = \sum_{(u, v) \in \dgm_k(n^{1/d}\Xn)}\!\!\!\!\!\delta_{(u,v)}.\nn
    }
    {Denoting by $B_{s, t} = [0, s] \times [t, \infty)$ } the rectangular set in the {upper-left half-plane}, the persistent Betti numbers are equivalently given by
    \eq{
    \beta_k^{s, t}(\K(n^{1/d}\Xn)) = \Big\langle{\mathbbm{1}({B_{s, t}}), \dgm_k(n^{1/d}\Xn)}\Big\rangle =\!\!\!\!\!\!\!\!\! \int\limits_{0 \le u < v \le \infty}\!\!\!\!\!\!\!\!\mathbbm{1}\qty\big( (u, v) \in B_{s, t} ) \  d\Psi\qty\big[\dgm(n^{1/d}\Xn)](u, v).\nn
    }
    As a direct consequence, \Fequivalence{} also implies that the distribution of $\inner{\phi, \dgm_k(n^{1/d}\Xn)}$ is the same for all $\pr \in \P$ and for all piecewise constant functions $\phi(\cdot)$ on the space of persistence diagrams. 
  \end{enumerate}
  \label{remark:betaequivalence}
\end{remark}

We conclude this section with a numerical illustration to demonstrate the topological inference in a family of distributions admitting \Fequivalence{} (and, thereby, also \betaequivalence{}).


{

\begin{example}\label{ex:biv-chi}

  \begin{figure}
    \centering
    \begin{subfigure}[T]{0.32\linewidth}
      \includegraphics[width=\linewidth]{\Root/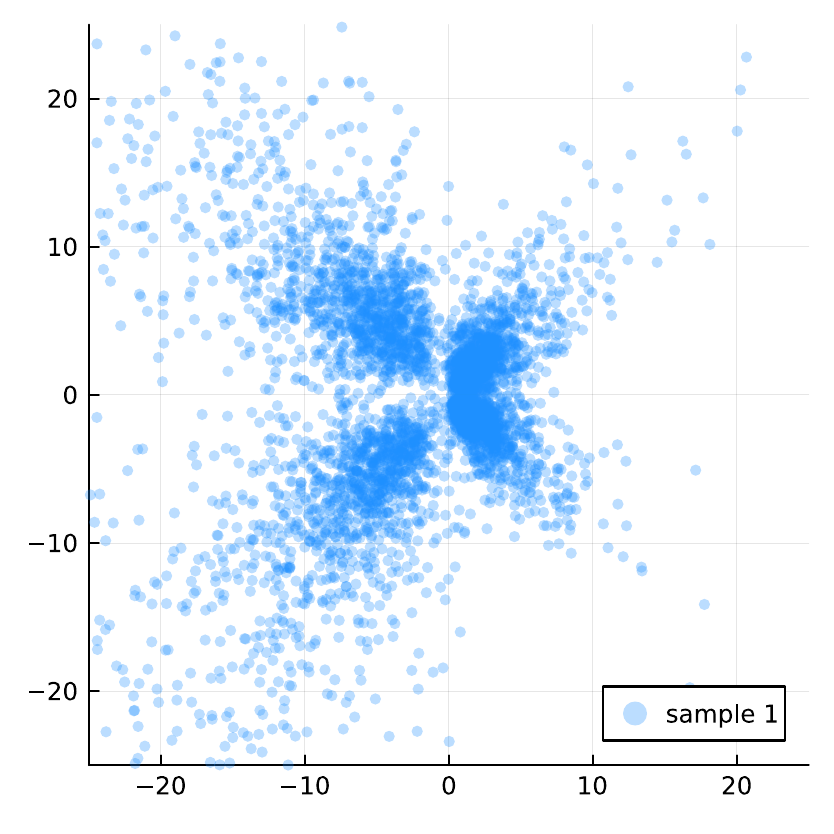}
      \caption{$\thetav = (0.46, 0.47, 0.03, 0.04)$}\label{fig:biv-chi-1-a}
    \end{subfigure}
    \begin{subfigure}[T]{0.32\linewidth}
      \includegraphics[width=\linewidth]{\Root/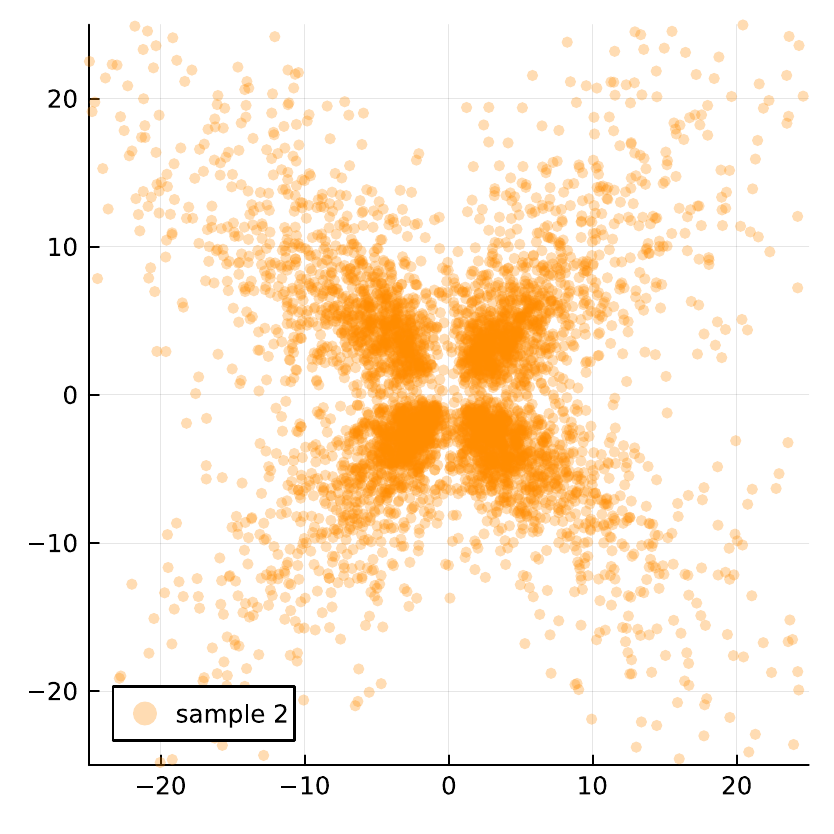}
      \caption{$\thetav = (0.17, 0.29, 0.21, 0.24)$}\label{fig:biv-chi-1-b}
    \end{subfigure}
    \begin{subfigure}[T]{0.34\linewidth}
      \includegraphics[width=\linewidth]{\Root/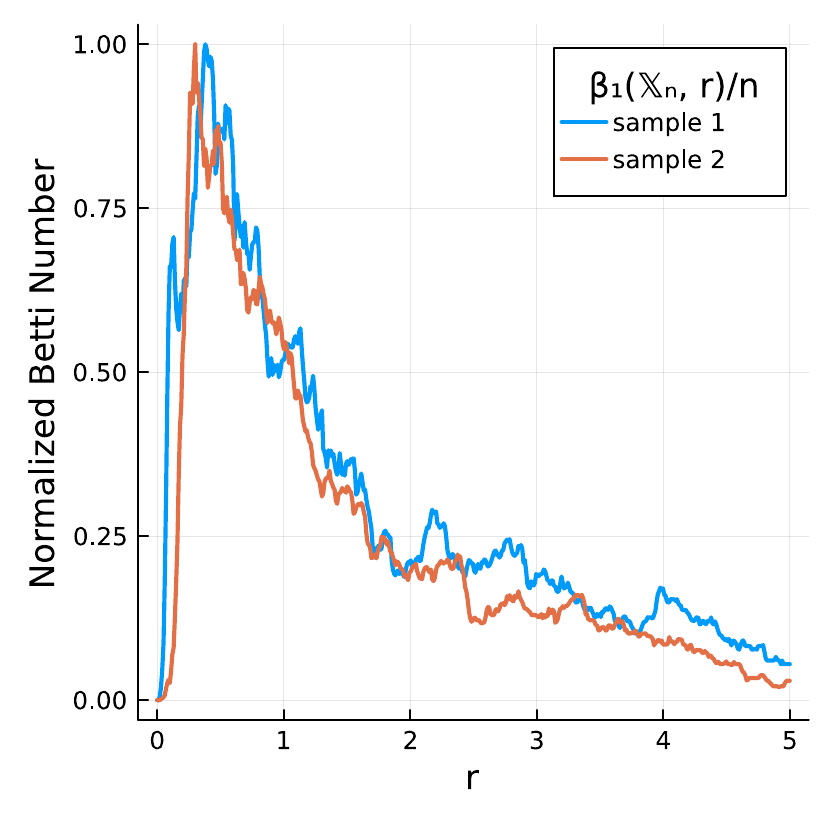}
      
      \vspace*{-1em}\caption{$r \mapsto \beta_1(\Xn, r)/n$}\label{fig:biv-chi-1-c}
    \end{subfigure}
    \caption{Scatterplot and Betti curve for point clouds $\Xn$ obtained from two different distributions in $\pb{\fthetav: \thetav \in \Theta}$. (Left) $\Xn \sim \fthetav$ for $\thetav = (0.46, 0.47, 0.03, 0.04)$ in blue, (Center) $\Xn \sim \fthetav$ for $\thetav = (0.17, 0.29, 0.21, 0.24)$ in orange, and (Right) the (normalized) Betti curve $r \mapsto \beta_1(\Xn, r)/n$ for the two point clouds.}
    \label{fig:biv-chi-squared-1}
  \end{figure}

  \begin{figure}
    \centering
    \begin{subfigure}[T]{0.48\linewidth}
      \includegraphics[width=\linewidth]{\Root/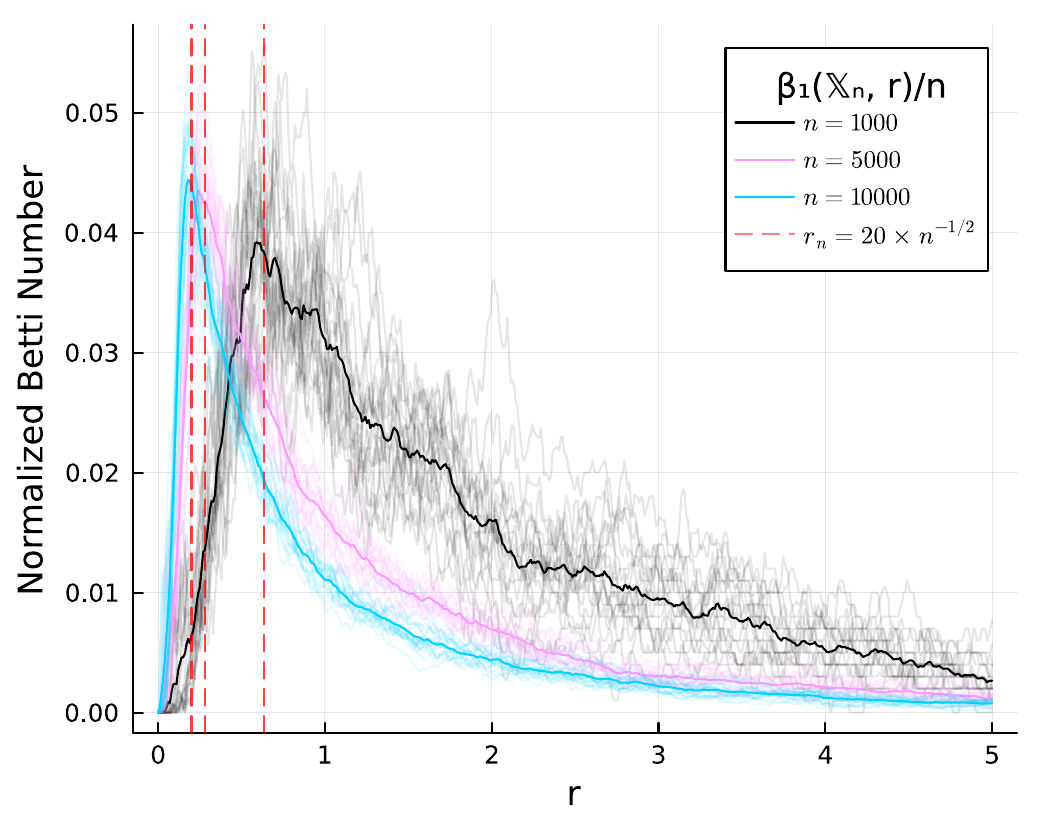}
      \caption{$r \mapsto \beta_1(\Xn, r)/n$ for varying $\thetav$ but fixed $n$.}
      \label{fig:biv-chi-2-a}
    \end{subfigure}
    \begin{subfigure}[T]{0.48\linewidth}
      \includegraphics[width=\linewidth]{\Root/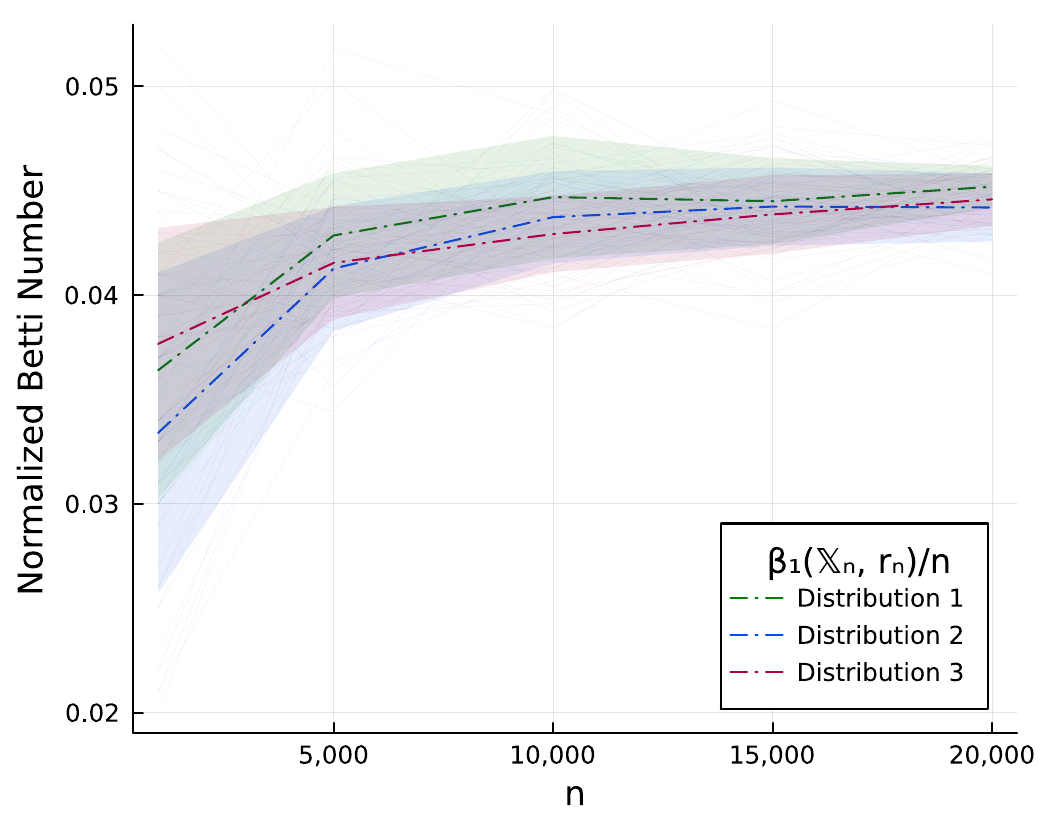}
      \caption{$n \mapsto \beta_1(\Xn, r_n)/n$ for fixed $\thetav$ and varying $n$.}
      \label{fig:biv-chi-2-b}
    \end{subfigure}
    \caption{Betti curves and the Betti numbers in the thermodynamic regime for $\pb{\fthetav: \thetav \in \Theta}$ from Example~\ref{ex:biv-chi}.}
    \label{fig:biv-chi-2}
  \end{figure}

  Let $\Theta = \pb{(\theta_1, \theta_2, \theta_3, \theta_4) \in \R^4_{+}: \theta_1 + \theta_2 + \theta_3 + \theta_4 = 4}$, and let $\fthetav$ be a probability density function on $\X = \R^2$, given by
  \eq{\label{eq:biv-chi}
  \ftheta(x_1, x_2) = \begin{cases}
    g(\phantom{+}{x}/{\sqrt\theta_1}, \  \phantom{+}{y}/{\sqrt\theta_1}) / 4, & \text{if } x \ge 0, y \ge 0\\
    g({-x}/{\sqrt\theta_2}, \ \phantom{+}{y}/{\sqrt\theta_2}) / 4, & \text{if } x < 0, y \ge 0\\
    g(\phantom{+}{x}/{\sqrt\theta_3}, \ {-y}/{\sqrt\theta_3}) / 4, & \text{if } x \ge 0, y < 0\\
    g({-x}/{\sqrt\theta_4}, \ {-y}/{\sqrt\theta_4}) / 4, & \text{if } x < 0, y < 0\\
  \end{cases},
  }
  
  where $g \sim \text{Biv-}\chi^2(10, 10, 0)$ is the probability density function of the bivariate $\chi^2$ distribution with parameters $n=p=10$ and $m=0$ \citep[Eq.~4.5]{gunst1973density}. In Example~\ref{ex:mass_transport_2}, we show that $\pb{\fthetav: \thetav \in \Theta}$ admits \Fequivalence{}, and, therefore, from Lemma~\ref{lem:inclusion} also admits \betaequivalence{}. Figures~\ref{fig:biv-chi-1-a}~and~\ref{fig:biv-chi-1-b} show the scatterplot of samples $\Xn$ obtained from two different distributions in this family. The $1^\textup{st}$ order (normalized) Betti curves $r \mapsto \beta_1(\K(\Xn, r)) / n$, in Figure~\ref{fig:biv-chi-1-c}, show that that their topological summaries are almost indistinguishable. 
  
  In Figure~\ref{fig:biv-chi-2-a}, for $k=20$ random parametrizations $\pb{\thetav_i: 1 \le i \le k}$ of $\fthetav$, we plot the Betti curves $\textsf{curve}_i: r \mapsto \beta_1\qty(\Xn(\theta_i), r)/n$ for $\Xn(\thetav_i)$ when $n \in \pb{1000, 5000, 10000}$. For fixed $n$, the (pointwise) mean Betti curve for the $k$ individual curves is highlighted in bold. The plot shows that as the sample size $n$ increases, the individual curves are harder to distinguish, and, shed light on the indistinguishability of the parameters in the asymptotic setting. In addition, the vertical dashed line in Figure~\ref{fig:biv-chi-2-a} highlights the region of maximum variability, which was empirically observed to correspond to $r_n = 20 \times n^{-1/2}$ in the thermodynamic regime. 
  
  To further investigate this region, we choose $k=3$ random parametrizations $f_{\thetav_i}$ for $i \in \qty{1, 2, 3}$. For $n$ ranging from $1,000$ to $20,000$, and for each combination of $n$ and $\thetav_i$ we generate $30$ realizations of the samples from $f_{\thetav_i}$, i.e., $\Xn^{(j)}(\thetav_i)$  for $j \in \qty{1, 2, \dots, 30}$, and compute their normalized Betti numbers in the thermodynamic regime, $\textsf{b}(n, i, j) := \beta_1(\K(\Xn^{(j)}(\thetav_i), r_n)) / n$, with $r_n = 20 \times n^{-1/2}$. Figure~\ref{fig:biv-chi-2-a} shows the plot of $n$ vs. $\textsf{b}(n, i, j)$ for each combination of $i$ and $j$. For each fixed $\thetav_i$, the mean curve is highlighted in a bold dotted line with the shaded region illustrating one standard deviation. The figure shows that the topological summaries, in this example, are insufficient to distinguish the three parameters in a formal statistical setting. The accompanying code for reproducing the experiments is available at \href{https://github.com/sidv23/invariance}{\textup{\texttt{https://github.com/sidv23/invariance}}}.
\end{example}

}
\section{\Fequivalence{} I: Algebraic Perspective}
\label{sec:invariance2}

Having introduced the notion of \Fequivalence{}, and discussed its implications in the preceding section, a natural question arises: \textit{when does a family of distributions admit \Fequivalence{}?} In this section, we impose an algebraic structure on the parametric model $\P(\Theta)$ in order to provide a general template for characterizing distributions that admit \Fequivalence{}. We begin by motivating the choice of imposing an algebraic structure by means of the following prototypical example. 

\begin{example}[Location and scale families]
  For $\X = \R^d$ and a fixed density function $f_{\zerov}$ on $\R^d$, consider the location family of distributions, $\P_{loc} = \qty{\fthetav(\xv) = f_{\zerov}(\xv - \thetav) : \thetav \in \R^d}$, if $\xtheta \sim \fthetav$, by a standard transformation of random variables it follows that ${(\xtheta - \theta) \distas \Xv_{\zerov} \sim f_{\zerov}}$. Therefore, $\fthetav(\xtheta) = f_{\zerov}(\xtheta - \thetav) \distas f_{\zerov}(\Xv_{\zerov})$, which does not depend on $\theta$. It follows that $\P_{loc}$ admits \Fequivalence{}.

  \smallskip

  However, the scale family of distributions, $\P_{scale} = \qty{g_\theta(\xv) = \theta\inv f_1(\xv / \theta) : \theta \in \R_+ }$ for a fixed density function $f_1$ does not admit \Fequivalence{}. If $\ytheta \sim \gtheta$, then $\ytheta/\theta \distas \Yv_1 \sim f_1$. But, the distribution of $\gtheta(\ytheta)$ is the same as the distribution of $\theta\inv f_1(\Yv_1)$, which clearly depends on $\theta \in \R_+$. Therefore, $\P_{scale}$ does not admit \Fequivalence{}.
\end{example}

It may be argued that it is natural to expect topological invariants to be insensitive to translations but sensitive to scaling and dilation. However, the next example illustrates a nontrivial family of distributions which admit \Fequivalence{}.

\begin{figure}
  \centering
  \includegraphics[width=0.9\linewidth]{\Root/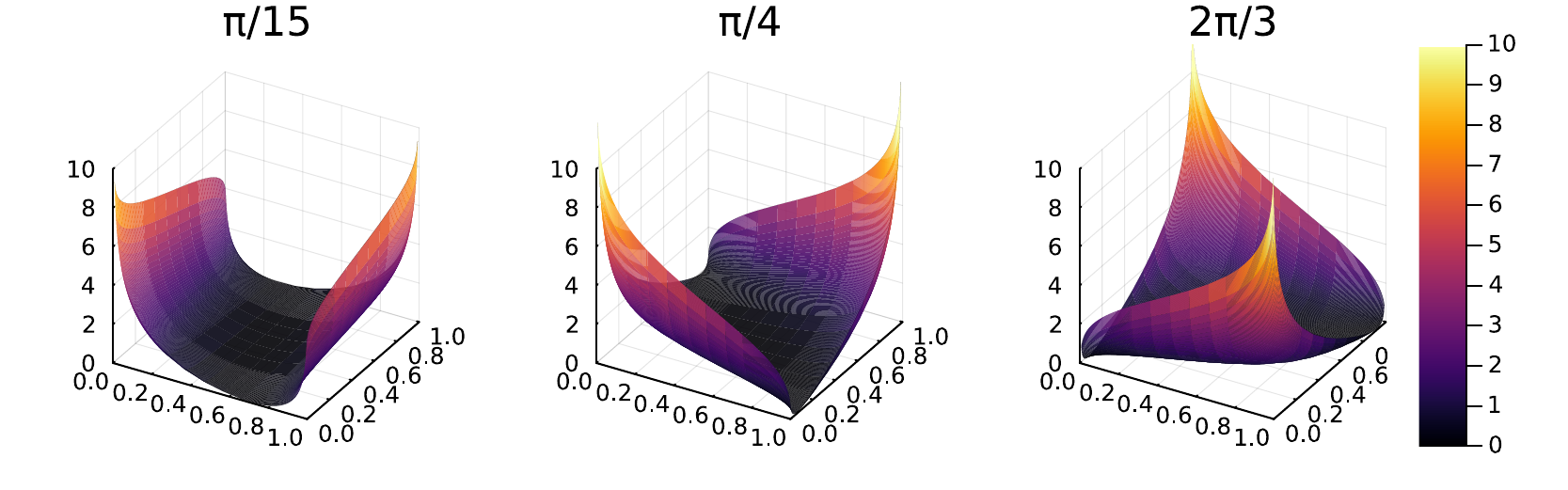}
  \caption{Illustration $\ftheta(x, y)$ from Example~\ref{ex:motivating} for $\theta \in \{\pi/15, \pi/4, 2\pi/3\}$}
  \label{fig:motivating}
\end{figure}

\begin{example}[Motivating example]\label{ex:motivating}
  Let $\X = [0,1]^2$ and $\Theta = [0, 2\pi]$, consider the family of distributions $\P(\Theta) = \pb{\ftheta: \theta \in \Theta}$ given by 
  $$
  \ftheta(x, y) = \qty(\cos\theta \cdot \Phi\inv(x) + \sin\theta \cdot \Phi\inv(y))^2 \mathbbm{1}((x, y) \in \X),\nn
  $$
  where $\Phi\inv(\cdot)$ is the inverse CDF of a standard Normal distribution $\mathcal{N}(0, 1)$. See Figure~\ref{fig:motivating} for an illustration. It is easy to verify that $\ftheta$ is a well-defined probability density function. Let $(X_\theta, Y_\theta) \sim \ftheta$ and consider the transformation $(x, y) \mapsto (\Phi\inv(x), \Phi\inv(y)) =: (u, v)$. A simple transformation shows that the density function for the transformed random vector $(U_\theta, V_\theta) = (\Phi\inv(X_\theta), \Phi\inv(Y_\theta))$ is given by
  \eq{
    f_\theta(u, v) = \qty( \cos\theta\cdot u + \sin\theta\cdot v )^2 (2\pi)\inv e^{-\half (u^2 + v^2)}.\nn
  }
  By further transformation $(u, v) \mapsto \qty\big(u\cos(\theta) + v\sin(\theta), -u\sin(\theta) + v\cos(\theta)) =: (r, s)$, the density function for the random vector $(R_\theta, S_\theta)$ becomes
  \eq{
    f_\theta(r,s) = (2\pi)\inv r^2 e^{-{r^2}/{2}} \cdot e^{-{s^2}/{2}},\nn
  }
  from which we can see that the distribution of $(R_\theta, S_\theta)$ is free of $\theta$ and $R_\theta \indep S_\theta$. Therefore, by marginalizing $(R_\theta, S_\theta)$ and then squaring, the desired distribution for $Z_\theta = \ftheta(X_\theta, Y_\theta) = R_\theta^2$ is 
  \eq{
    f_{Z_\theta}(z) = \qty({z/2\pi})^{-1/2} e^{-z/2},\nn
  }
  which shows that $Z_\theta \sim \Gamma(3/2, 1/2)$ for all $\theta \in \Theta$. Therefore, $\P(\Theta)$ admits \Fequivalence{}. 
\end{example}

A key step in establishing \Fequivalence{} in Example~\ref{ex:motivating} is showing that the distribution of $(R_\theta, S_\theta)$ is free of $\theta$. In particular, the map $(u, v)\!\mapsto\!(r,s)$ can be represented by the action of an element of the rotation group, i.e., $(r, s) = g_\theta(u, v)$, where $g_\theta \in \sor{2}$; and the $(u^2 + v^2)$ term in $f_\theta(u, v)$ ensures that the action of $g_\theta$ remains constant on orbits of $g_\theta$. This suggests that the algebraic structure underlying the family of distributions (if any) plays a key role in establishing their invariance. Lie groups provide a natural framework for studying such properties. We begin by reviewing some basic facts about group invariance in statistics. The presentation here closely follows \cite{wijsman1990invariant} and \cite{eaton1989group}.

Given a group $\G \defeq \pa{\G,*}$ acting on a space $\X$, its \emph{orbit} is given by $\G\xv = \pb{g\xv : g \in \G}$; the \emph{stabilizer} with respect to $\xv$ is $\G_{_\xv} = \pb{g\in \G : g\xv=\xv}$. When $\G$ acts bijectively on $\X$, a function $T : \X \rightarrow \Tau$ is $\G$-invariant if $T(\xv) = T(g\xv)$ for all $g\in \G$ {and a space $\Tau$}.

\begin{definition}[Maximal Invariant]
  {A function $T: \X \rightarrow \Tau$ is {$\G$-maximal invariant} if it is constant on orbits, i.e., $T(g\xv) = T(\xv)$ for each $g \in \G$; and, it takes different values on different orbits, i.e., for each $\xv,\yv \in \X$ such that $T(\xv) = T(\yv)$, we have that $\yv \in \G\xv$.\label{def:maximal_invariant}}
\end{definition}

It follows from this definition that if $J: \Tau \rightarrow \mathcal{J}$, {from the space $\Tau$ to the space $\mathcal{J}$}, is any injective map, then the composition $J \circ T$ will also be $\G$-maximal invariant. Thus, maximal invariants for a group $\G$ are unique up to injective transformations. The relationship between a $\G$-invariant function and $\G$-maximal invariance is described in the following result.

\begin{proposition}[\citealp{eaton1989group}]
  Suppose $T : \X \rightarrow \Tau$ is $\G$-maximal invariant. A function $\phi : \X \rightarrow \Y$ is $\G$-invariant if and only if there exists $\kappa: \Tau \rightarrow \Y$ such that $\phi = \kappa \circ T$.
  \label{lem:maximal_invariant}
\end{proposition}

If $\Psi: \X \rightarrow \Y$ is a surjective function
taking elements from $\X$ to the space $\Y$,
then the action of $g\in \G$ on elements in $\X$ \textit{induces} an action on elements in $\Y$ via $\Psi$. The \textit{induced action} of $\G$ on $\Y$ is given by $g\yv \defeq \Psi(g\xv)$ for every $\xv\in \Psi\inv(\yv)$. This is well-defined whenever $\Psi(\xv_1) = \Psi(\xv_2)$ implies that $\Psi(g\xv_1) = \Psi(g\xv_2)$ for each $g\in \G$. We say that $\Psi: \X \rightarrow \Y$ is \emph{$\G$-compatible} when the induced action of $\G$ on $\Y$ is well-defined. 

With this background, let $\pr$ be a fixed distribution on $\X$ and $\Xv \sim \pr$. The action of an element $g \in \G$ on $\X$ induces a transformation on $\Xv$ to a new random variable $g{\Xv}$ taking values in $\pa{g{\X},\borel\pa{g{\X}}}$. With a slight abuse of notation, let $g_{\#}\pr$ denote the distribution of $g\Xv$. If the elements of $\G$ are indexed by a parameter $\theta \in \Theta$, i.e., $\G(\Theta) = \pb{\gtheta: \theta \in \Theta}$, then the action of $\G$ on $\Xv$ induces a family of distributions $\P(\Theta) = \pb{\pr_\theta = {\gtheta}_{\#}\pr : \theta \in \Theta}$. The following result establishes necessary and sufficient conditions for $\P(\Theta)$ to admit \Fequivalence{} for a wide class of distributions.

\begin{theorem}[Group Invariance]
  Suppose $\Psi: \X \rightarrow \Y$ is differentiable and bijective, $\mathcal{G}(\Theta)$ is a group of Borel-measurable isometries acting on $\Y$, and $T: \Y \rightarrow \Tau$ is $\mathcal{G}(\Theta)$-maximal invariant. Define the family of probability distributions $\P(\Theta) = \pb{\ftheta : \theta \in \Theta}$ on $\X$ by
  \eq{
  \ftheta(\xv) = \phi\pa{\gtheta \circ \Psi\pa{\xv}},\nonumber
  }
  where $g_\theta\in\mathcal{G}(\Theta)$ and $\phi: \X \rightarrow \R_{\geq 0}$ is some function which ensures that $\ftheta$ is a valid density. Then,

  {
  \begin{enumerate}[label=\textup{(\roman*)}]
      \item $\P(\Theta)$ admits \Fequivalence{} if and only if there exists $\zeta: \Tau \rightarrow \R$ such that
    $$
    {\det\pa{\mathbf{D}{\Psi\inv}\pa{\yv}} = \zeta\pa{T\pa{\yv}}}.
    $$

    \begin{minipage}{0.9\textwidth} 
      \item If $\G(\Theta) = \mathop{\times}_{i=1}^{m}{\G_i(\Theta_i)}$ where $\Theta = \Theta_1 \times \Theta_2 \times \dots \times \Theta_m$, then $\P(\Theta)$ admits \Fequivalence{} if and only if there exists a sequence of $\G_i(\Theta_i)$-compatible functions $T_i: \Y_{i-1} \rightarrow \Y_{i}$ with $\Y_0 = \Y$ and a function $\zeta: \Y_m \rightarrow \R$ such that each $T_i$ is $\G_i(\Theta_i)$-maximal invariant and
      \eq{
      {\det\pa{\mathbf{D}{\Psi\inv}\pa{\yv}} = \zeta\pa{T_m \circ T_{m-1} \circ \dots \circ T_1\pa{\yv}}}.\nn
      }
  \end{minipage}
  \end{enumerate}
  }
  \label{thm:group_invariance}
\end{theorem}
We defer the proof to Section~\ref{proof:thm:group_invariance} and illustrate some examples of Theorem \ref{thm:group_invariance} in the remainder of the section. We begin by illustrating a multivariate generalization of Example~\ref{ex:motivating} in the context of Theorem~\ref{thm:group_invariance}. 

\begin{example}
  Let $\xi \sim F$ be random variable on $\R$ with density $f(x) = F'(x)$ such that $\E(\xi)=0$ and $\E(\xi^{2})=1$. For $\X = \pc{0,1}^d$ and $\Theta \in \S^{d-1}$, let $\P(\Theta) = \pb{\fthetav: \theta \in \Theta}$ be given by
  \eq{
  \fthetav(\xv) = {\qty\big({\thetav^{\top}{\Fv}\inv(\xv)})^2} \mathbbm{1}(\xv \in \X),
  \label{eq:sod_density}
  }
  where, for brevity, ${\Fv}\inv(\xv) \defeq \qty\big({F\inv(x_1),F\inv(x_2),\dots,F\inv(x_d)})^\top \in \R^{d}$ is the ``vectorized'' inverse CDF of $F$. Then, the family of distributions $\P(\Theta)$ admits \Fequivalence{} if and only if $\xi \sim \mathcal{N}(0,1)$.

  First note that $\ftheta$ is a well-defined density. To see this, taking $\yv = \Fv\inv(\xv)$, we have
  {
  \eq{
    \int_{\X}\ftheta(\xv)d\xv 
    &= \int_{\R^d} (\theta\tr\yv)^2 \prod_{i=1}^d f(y_i) \  d\yv\nn\\
    &= \sum_{j=1}^d\sum_{k=1}^d \int_{\R^d} \theta_j\theta_k y_j y_k \prod_{i=1}^d f(y_i)  \ d\yv \nn\\
    &= \sum_{j=1}^d \int_{\R^d} \theta_j^2 y^2_j \prod_{i=1}^d f(y_i)  \ d\yv + \sum_{1 \le j < k \le d} \int_{\R^d} 2 \cdot \theta_j\theta_k y_j y_k \prod_{i=1}^d f(y_i)  \ d\yv \nn\\
    &\stackrel{\textup{(i)}}{=} \sum_{j=1}^d \theta_j^2 + 0 = 1,\nn
  }
  where (i) uses the fact that $\E(\xi) = 0$ and $\E(\xi^2) = 1$. 
  }

  Let $\G = \sor{d} = \qty{ \gtheta \in GL(\R, d) : \gtheta\inv = \gtheta\tr, \gtheta\tr e_1 = \theta \in \S^{d-1} }$ be the group of rotations on ${\Y=\R^d}$ where $e_1 = (1, 0, \dots, 0)^\top \in \R^d$. {The function $T(\yv) = \norm{\yv}^2$ is $\G$-maximal invariant since $T(\yv_1) = T(\yv_2)$ if and only if $\yv_2 = \gtheta \yv_1$ for some $\gtheta \in \G$}. As per Theorem~\ref{thm:group_invariance}, \eref{eq:sod_density} is equivalently written as $\ftheta(\xv) = \phi(\gtheta \circ \Psi(\xv))$ where 
  \eq{
    \Psi(\xv) =  \Fv\inv(\xv), \quad and \quad \phi(\yv) = (\yv\tr e_1)^2.\nn
  }
  The Jacobian for $\Psi\inv$ is given by $\norm{\mathbf{D}{\Psi\inv}(\yv)} = \prod_{i=1}^n f(y_i)$. From Theorem~\ref{thm:group_invariance} it follows that $\P(\Theta)$ admits \Fequivalence{} if and only if there exists some function $\zeta_d: \R_+ \rightarrow \R_+$, which may implicitly depend on $d$, such that $\prod_{i=1}^n f(y_i) = \zeta_d(\norm{\yv}^2)$. We show, using Lemma~\ref{lemma:normal_density} in Appendix~\ref{supplementary}, that this is satisfied only when $\xi \sim \mathcal{N}(0, 1)$. When $d=2$, this recovers the family of distributions illustrated in Example~\ref{ex:motivating}.
  \label{ex:orthogonal}
\end{example}

To illustrate a family of distributions generated by a subgroup of transformations, consider the following variant of Example~\ref{ex:orthogonal}.

\begin{example}
  For $p+q=d$, consider $\G=\sor{p} \times \sor{q}$ acting as a subgroup of $\sor{d}$ on $\Y=\R^d$, i.e., for every $\gtheta \in \G$ it follows that $\gtheta\tr e_1 = \theta_p \oplus \theta_q$ where $\theta_p \in \S^{p-1}$ and $\theta_q \in \S^{q-1}$. Let $\Phiv_p\inv: [0,1]^p \rightarrow \R^p$ and $\Phiv_q\inv: [0,1]^q \rightarrow \R^q$ be the ``vectorized'' inverse CDFs of $\N(0,\sigma^2_p)$ and $\N(0,\sigma^2_q)$. For $\X = [0,1]^d$, consider $\P(\Theta) = \pb{\ftheta: \theta \in \S^{p-1} \times \S^{q-1}}$ given by
  \eq{
    \ftheta(\xv) = \kappa(p, q)\qty(\theta_p\tr\Phiv_{p}\inv(\xv_p) + \theta_q\tr\Phiv_{q}\inv(\xv_q))^2 \ \mathbbm{1}(\xv_p \oplus \xv_q \in \X),
  }
  where $\xv = \xv_p \oplus \xv_q$ and $\kappa(p, q)$ is a fixed normalizing constant free of $\theta$. As before, $\ftheta(\xv) = \phi(\gtheta \circ \Psi(\xv))$ where
  \eq{
    \Psi(\xv) = \qty(\Phiv_{p}\inv(\xv_p), \Phiv_{q}\inv(\xv_q)), \quad and \quad \norm{\mathbf{D}{\Psiv\inv}(\yv)} = \f{\sigma_p^p \sigma_q^q}{(2\pi)^{d/2}} \exp\pa{-\f{\norm{\yv_p}^2}{2\sigma_p^2} -\f{\norm{\yv_q}^2}{2\sigma_q^2}}.\nn
  }
  In the context of Theorem~\ref{thm:group_invariance} let $\Y_1 = \R^{p+2}$ and $\Y_2=\R^2$, and consider $T_1: \Y \mapsto \Y_1$ and $T_2: \Y_1 \mapsto \Y_2$ given by
  \eq{
    T_1: \yv_p \oplus \yv_q \mapsto \yv_p \oplus \norm{\yv_q}^2, \quad and \quad T_2: \yv_p \oplus z \mapsto \norm{\yv_p}^2 \oplus z.\nn
  }
  Clearly, $T_1$ is $\sor{q}$-maximal invariant and $T_2$ is a maximal invariant w.r.t. the induced action of $\sor{p}$ on $\Y_1$. Therefore, for $\zeta: \R^2$ given by
  \eq{
    \zeta(x,y) = \f{\sigma_p^p \sigma_q^q}{(2\pi)^{d/2}} \exp\pa{-\f{x}{2\sigma_p^2} -\f{y}{2\sigma_q^2}},\nonumber
  }
  it follows that $\norm{\mathbf{D}{\Psiv\inv}(\yv)} = \zeta(T_2 \circ T_1 (\yv))$, implying that $\P(\Theta)$ admits \Fequivalence{}.
\end{example}

The preceding two examples provide an alternate characterization of the Normal distribution through the lens of \Fequivalence{} w.r.t.~rotational transformations. The next example illustrates \Fequivalence{} for a family of distributions generated by an unconventional group of transformations. 

\begin{example}
  For a fixed shape parameter $\kappa = 0.5$, consider the family of bivariate Weibull distributions on $\R^2_+$ given by
  \eq{
  \ftheta(x,y) = \f{1}{4\sqrt{xy}} \exp\pa{-\theta \sqrt{x} - \f{\sqrt{y}}{\theta}}.\nonumber
  }
  For $\Theta=\R_+$, it follows that $\P(\Theta) = \pb{\ftheta(x,y) : \theta \in \Theta}$ admits \Fequivalence{}. We verify this as per Theorem \ref{thm:group_invariance}. The functions $\Psi$ and $\phi$ are $\Psi(x,y) = (\sqrt{x},\sqrt{y})$, and  $\phi(x,y) = \exp\pa{-(x+y)}/4xy$. The group action may be identified as follows: {let $\G$ be a subgroup} of $GL(\R, 2)$ consisting of elements
  \eq{
  \gtheta = \begin{pmatrix}\setlength{\arraycolsep}{10pt}
  \theta & 0\\
  0 & {1}/{\theta}
  \end{pmatrix},\nn
  }
  for $\theta \in \R_+$. It is easy to verify that $T(x,y) = xy$ is $\G$-maximal invariant. The density can now be expressed in the form $\ftheta(x,y) = \phi\pa{\gtheta \circ \Psi(x,y)}.\nn$
  It follows that $\Psi\inv(x,y) = ({x^2,y^2})$ and 
  $$
  \det\pa{\mathbf{D}{\Psi\inv}(x,y)} = 4xy = 4T(x,y).
  $$
  Hence, by Theorem \ref{thm:group_invariance}, $\P(\Theta)$ admits \Fequivalence{}.
\end{example}
\section{\Fequivalence{} II: General Cases}
\label{sec:invariance3}

While the results in Section \ref{sec:invariance2} examined necessary and sufficient conditions for \Fequivalence{} by enforcing some algebraic structure on the family of distributions, the objective of this section is to relax these requirements and, instead, exploit the structure underlying the support of the distributions, $\X$. Before we present the main results, we introduce the main tools we employ:
\begin{enumerate}[itemsep=-2pt,label=\textup{(\roman*)}]
  \item The {\textit{modular character of a measure}} $\mu$, denoted by $\Psi_\mu$, and
  \item A fiber bundle representation $\e = (\X, \z, \Y, \pi, G)$ of the underlying space $\X$.
\end{enumerate}

\smallskip

\noindent {\bfseries Modular Character. \hquad} Consider the set of diffeomorphisms $\Delta(\X)$, given by
$$
  \Delta(\X) \defeq \pb{\phi \in \textup{Diff}\pa{\X} : \absdetj{\phi(\xv)} = \absdetj{\phi(\xv')}, \ \forall \xv,\xv' \in \X}.
$$
In other words, $\Delta(\X)$ comprises of smooth maps from $\X$ to itself such that the Jacobian of the map does not depend on the specific location where the transformation is made, e.g., when $\X=\R^d$ the set $\Delta(\X)$ {is the group of rigid transformations on $\R^d$, $E(d)$}. The elements of $\Delta(\X)$ form a subgroup of transformations with respect to $\textup{Diff}\pa{\X}$. The change in measure induced by diffeomorphic transformations of a space $\X$ w.r.t.~the measure $\mu$ is given by its modular character.

\begin{definition}[Modular Character]
  Given a measure $\mu$ on the space $\pa{\X,\borel(\X)}$, a function $\Psi$ is defined to be the modular character of $\mu$ if for each  $\phi \in \Delta(\X)$ with $\yv = \phi(\xv)$ we have that
  \eq{
    \mu(d\yv) = \Psi\pa{\absdetj{\phi}} \mu(d\xv).\nonumber
  }
\end{definition}

For example, when $\X = \R^d$ and $\mu = \lambda_d$, the $d$-dimensional Lebesgue measure, the modular character of $\mu$ is given by $\Psi(x) = x$. Observe that for any full-rank linear map $\phi \in GL(\R, d)$ such that $\yv = \phi(\xv)$, we have $\mu(d\yv) = d\yv = \absdetj{\phi} d\xv = \absdetj{\phi} \mu(d\xv)$. A nontrivial example of the modular character is illustrated in Example~\ref{ex:spherical}.

\begin{remark}
  The modular character is closely related to the notion of tensor-density (see, for instance, \citealt{schouten1954ricci}). When $G \le \Delta(\X)$ is a locally compact group continuously acting on $\X$ from the left, i.e., $(g, \xv) \mapsto g \cdot x$ for every $g \in G$, then $\Psi_\mu$ is also called the relatively invariant multiplier \citep{eaton1989group}.
\end{remark}

\smallskip

{
  \noindent{\bfseries Fiber Bundle Representation. \hquad} Following the convention in \cite{steenrod1999topology}, when $\X$ is a compact $d$-dimensional $\mathcal{C}^1$-manifold, suppose $\X$ admits a fiber bundle representation ${\e \defeq \pa{\X, \z, \Y, \pi, G}}$, i.e., $\X$ is the total space, $\z$ is the base space with the bundle projection $\pi$ defined by the continuous surjective map $\pi:\X \rightarrow \z$, and $G$ is a topological group\footnote{We use $G$ to distinguish the structure group, in this section, from the groups, $\G$, in Section~\ref{sec:invariance2}.} which acts on the canonical fiber space $\Y$. A collection $\pb{(V_j, \psi_j) : j \in J}$ is called a \textit{local trivialization} of $\e$ and serves as the coordinate charts for the fiber bundle, i.e., $\pb{V_j : j \in J}$ is an open cover of $\z$ and for each $j \in J$ the map $\psi_j : V_j \times \Y \rightarrow \pi\inv(V_j)$ is a diffeomorphism which guarantees that locally, in the neighborhood $V_j \subset \z$, the fiber $\pi\inv(V_j)$ looks like the product $V_j \times \Y$. In particular, the map $\psi_{j, \zv} = \psi_j(\zv, \cdot): \Y \rightarrow \Y_{\zv} := \pi\inv(\zv) $ is a diffeomorphism for each $\zv \in V_j \subseteq \z$. 
  
  Furthermore, for every $i, j \in J$ and all $\zv \in V_i \cap V_j$, the map $\psi\inv_{j, \zv} \circ \psi_{i, \zv}: \Y \rightarrow \Y$ should coincide with an element $g_{ji,\zv}$ of the structure group $G$, and $\e$ is represented by the following commutative diagram:
}

\begin{center}
  \begin{tikzcd}
    \pi^{-1}(V_j) \arrow[r, "\psi_j\inv"] \arrow[d, "\pi"'] & V_j \times \mathcal Y \arrow["g \in G"', loop, distance=2em, in=35, out=325] \arrow[ld, "{(\boldsymbol z, \boldsymbol y) \mapsto \boldsymbol z}"] \\
    V_j                                                 &
  \end{tikzcd}
\end{center}

{
From \citet[Section~2.4]{steenrod1999topology}, if $\pb{(U_i, \eta_i) : i \in I}$ is another local trivialization for $\e$, then the two local trivializations are \textit{equivalent} in the sense that for every $\zv \in U_i \cap V_j$ there exists $g\ijz \in G$ such that
\eq{\label{eq:charts}
  g\ijz = \psi\inv\jz \circ \eta\iz.
}
In other words, the structure group $G$ determines the change of coordinates from $\pb{(V_j, \psi_j) : j \in J}$ to $\pb{(U_i, \eta_i) : i \in I}$.




If the base space $\z$ and the fiber $\Y$ are endowed with measures $\mu$ and $\nu$ respectively, they induce a \textit{local product measure} $\lambda = \nu \otimes_{loc} \mu$ on the space $\X$ \citep{goetz1959measures}. Specifically, for a measurable set $A \subset \X$ such that $\pi(A) \subset V_j \subseteq \z$, the induced measure $\lambda$ is given by
\eq{
  \lambda(A) = \pa{\nu \otimes\loc \mu }(\psi_j\inv(A)) = \int_{\pi(A)} \nu_{\zv}(\pi\inv(\zv) \cap A) \mu(d\zv),
  \label{eq:product_measure}
}
where for $\xv \in \X$ and $\zv = \pi(\xv)\in V_j \subseteq \z$, the measure $\nu_{\zv} \defeq \pa{\psi\jz}_{\#}\nu$ is the pushforward of $\nu$ on the space $\pi\inv(\zv)$ via the diffeomorphism $\psi\jz$, i.e., $\nu_\zv(B) = \nu(\psi\jz\inv(B))$ for all measurable $B \subseteq \pi\inv(\zv)$. Furthermore, for $f\!\in\!L^1(\X,\lambda)$, \cite{goetz1959measures} characterizes the following version of \mbox{Fubini's theorem:}
\eq{
  \int_{\X}{f(\xv)\lambda(d\xv)} = \int_{\z} \int_{\pi\inv(\zv)}{f(\wv) \ \nu_{\zv}(d\wv) \mu(d\zv)}.
  \label{eq:fubini}
}

From \citet[Theorem~1]{goetz1959measures}, the measure $\lambda$ in \eref{eq:product_measure} exists only when $\nu$ is $G$-invariant. To see this, consider two local trivializations $\pb{(V_j, \psi_j): j \in J}$ and $\pb{(U_i, \eta_i): i \in I}$; for $\zv \in U_i \cap V_j$ and a measurable set $B \subseteq \pi\inv(\zv)$ from the fiber over $\zv$, the induced measure $\nu_\zv(B)$ should be invariant to the specific choice of the local trivialization, i.e., $\nu\qty\big(\eta\iz\inv(B)) = \nu\qty\big(\psi\jz\inv(B))$. From \eref{eq:charts}, this implies that
\eq{
  \nu\qty\big(\eta\iz\inv(B)) = \nu\qty\big(g\ijz \circ \psi\jz\inv(B)) = \nu\qty\big(\psi\jz\inv(B)),\nn
}
from which it follows that $\nu$ is $G$-invariant.
It is worth pointing out that when $\e$ is a flat bundle, i.e., $\X = \z \times \Y$, then $G = \pb{\textup{id}_{\Y}}$ and $\nu$ is always well-defined.

\textbf{Excess mass function.} \quad The final ingredient we require is an alternate characterization of \Fequivalence{} using \textit{excess mass functions}. They are defined as follows: For a probability distribution $\pr$ with density $f$ and $\Xv \sim \pr$, the \textit{excess mass function}, $\hat{f}$, is given by
\eq{
\hat{f}(t) = \pr\qty\big(f(\Xv) \ge t) = \int_\X \mathbb{1}(f(\xv) \ge t) f(\xv) d\xv, \quad \text{ for all } t \ge 0.\nn
}

Excess mass functions have been employed by, e.g., \cite{muller1991excess} \& \cite{polonik1995measuring}, for geometric inference in the framework of nonparametric statistics, and by \cite{bubenik2007statistical} to characterize the Betti-$0$ function.

The next lemma supplements Lemma~\ref{lem:inclusion} using excess mass functions.

\begin{lemma}
  For two probability density functions $f$ and $g$, $f \Feq g$ if and only if $\hat{f} = \hat{g}$. Moreover, if $\P$ is a family of distributions such that $\hat{f} = \hat{g}$ for all $f, g \in \P$, then $\P$ admits \betaequivalence{}.
  \label{lem:excess_mass}
\end{lemma}

In other words, for $\Xv \sim f$ and $\Yv \sim g$ we may replace the condition $f(\Xv) \distas g(\Yv)$ in Definition~\ref{def:Fequivalence} with the condition $\hat f = \hat g$. Lemma~\ref{lem:inclusion} formalizes the intuition that the topological summaries in the thermodynamic limit are capturing, at most, local information encoded in the superlevel sets of the probability density function.

With this background, the following result provides a sufficient characterization for distributions to admit \Fequivalence{} and is generated using structure underlying~$\X$.

\begin{theorem}
  Let $\e = \pa{\X,\z,\Y,\pi,G}$ be a fiber bundle representation of $\X$ with a local trivialization $\pb{(V_j, \psi_j) : j \in J}$ and compact $\Y$. For each $\xv  \in \X$ with $\zv = \pi(\xv) \in V_j$, let $f_\phi$ be given by
  \eq{
    f_\phi(\xv) = C \cdot g\qty\Big( \phi\qty\big( \psi\jz\inv(\xv), \zv ) ),
    \label{eq:fiber}
  }
  where
  \begin{enumerate}[topsep=0pt,itemsep=5pt,partopsep=1ex,parsep=1ex, leftmargin=2em,label=\textup{(\roman*)}]
    \item $C = 1 / \nu(\Y)$ is the density of the uniform distribution w.r.t. a measure $\nu$ on $\Y$,
    \item $g$ is the density of a probability distribution on $\z$ w.r.t. a base measure $\mu$ with modular character $\Psi_\mu$,
    \item $\phi : \Y \times \z \rightarrow \z$ is such that $\phi_\yv \defeq \phi(\yv,\cdot) \in \Delta(\z)$ for each $\yv \in \Y$.
  \end{enumerate}
  Then $\pb{f_\phi : \phi \in \Phi}$ admits \Fequivalence{} when
  \eq{
    \label{eq:jac_constraint}
    \Phi \subseteq \qty{ \phi: \int_{\Y} C \cdot \Psi_\mu\qty\big(\absdetj{\phi\inv(\yv, \cdot) }) \ d\nu(\yv) = 1 }.
  }\label{thm:nonlinear_invariance}
\end{theorem}
The proof is deferred to Section~\ref{proof:thm:nonlinear_invariance}, and we provide the intuition here. The density $g$ on the base space $\z$ together with the uniform distribution on the fiber $\Y$, induce a probability density $C \cdot g(\zv)$ locally in $\X$. The map $\phi(\yv, \cdot)$ is free to move mass on the base space $\z$, which, in turn, induces a different mass in the fiber over $\zv$. By Lemma~\ref{lem:excess_mass}, \Fequivalence{} holds when the excess mass function is preserved; this corresponds to the Jacobian constraint in \eref{eq:jac_constraint}. Theorem~\ref{thm:nonlinear_invariance} demonstrates that, in the thermodynamic limit, the topological summaries are effectively capturing, at most, only local information underlying the probability distribution.


While the density function $f_\phi$ in Theorem~\ref{thm:nonlinear_invariance} seems to depend on the choice of the local trivialization $\pb{(V_j, \psi_j) : j \in J}$, the following result shows that for a fixed $\phi \in \Phi$, the density function $f_\phi$ doesn't depend on the choice of the local trivialization used.

\newpage

\begin{proposition}
  Under the conditions of Theorem~\ref{thm:nonlinear_invariance}, let $\pb{(V_j,\psi_j) : j \in J}$ and $\pb{(U_i,\eta_i) : i \in I}$ be two local trivializations of $\e = (\X, \z, \Y, \pi, G)$. For a fixed $\phi$ and for $\xv \in \X$ with $\zv = \pi(\xv) \in U_i \cap V_j$, let $f_\phi(\xv)$ be the density given by \eref{eq:fiber}, and let $\ft_\phi(\xv)$ be given by
  \eq{\label{eq:fiber-alt}
    \ft_\phi(\xv) = C \cdot g\qty\Big( \phi\qty\big( \eta\iz\inv(\xv), \zv ) ).
  }
  Then, for every measurable $A \subset \X$,
  \eq{
    \int_{A}f_\phi d\lambda = \int_{A}\ft_\phi d\lambda.\nn
  } \label{prop:coordinate-equivalence}
\end{proposition}
In other words, Proposition~\ref{prop:coordinate-equivalence} ensures that two equivalent coordinate representations of the $\e$ induce the same probability density function $f_\phi$ on $\X$ for a fixed $\phi$, and the collection $\Phi$ in Theorem~\ref{thm:nonlinear_invariance} cannot be simply obtained by a coordinate transformation of $\e$.
\begin{remark}
  \begin{enumerate}[label=\textup{(\roman*)}]
    We highlight some salient observations regarding Theorem~\ref{thm:nonlinear_invariance} below.

    \item When $\Y = \G$ is a Lie group, $\mathscr{X}$ simplifies to become a principal $\G$-bundle and the bundle~projection $\pi : \X \rightarrow \X/\G$ projects each element in $\X$ to its orbit. Principal $\G$-bundles admit local cross-sections. Furthermore, when it admits a global-cross section, $\X$ admits the factorization ${\X = \X/\G \times \G}$, and the factorization of measure on $\X$ simplifies to the product of an invariant measure $\mu$ and an equivariant measure $\nu$ (see, for example, \citealt{kamiya2008star}). Lastly, when $\X = \Y \times \z$ is globally trivial, then the bundle charts simply become $\psi_\alpha = \textup{id}_{X}$, such that the induced-measure $\nu_{\zv}=\nu$ for each $\zv \in \z$, and \eref{eq:fubini} reduces to the familiar setting of Fubini's theorem.
          \item\label{remark:measure-zero} Suppose $\pr$ is a distribution on $\X$ with density $f$ w.r.t. a dominating measure $\lambda$, and $A \subset \X$ is a set of $\lambda-$measure zero. Let $\pr\big|_{\X \setminus A}$ be the restriction of $\pr$ on $\X \setminus A$ and let $f_{\X\setminus A}$ be its resulting density. Then, for all $t \ge 0$,
          $$
            \hat{f}(t) = \int_{\X} \mathbbm{1}\qty(f(\xv) \ge t) f(\xv) d\lambda(\xv) \stackrel{(i)}{=} \int_{\X\setminus A} \mathbbm{1}\qty(f(\xv) \ge t) f(\xv) d\lambda(\xv) = \hat{f}_{\X\setminus A}(t),
          $$
          where (i) follows by noting that $\lambda(A) = 0$. Therefore, $\hat{f} = \hat{f}_{\X\setminus A}$, indicating that \Fequivalence{} for families of distributions can be analyzed by excluding a set of measure zero from the space $\X$. In particular, if $\e = \qty(\X, \z, \Y, \pi, G)$ is a fiber bundle, then we may omit sets of measure zero from both $\z$ and $\Y$ since the dominating measure $\lambda$ on $\X$ is admits the disintegration $\lambda = \nu \otimes_{\textup{loc}} \mu$, for dominating measures $\nu$ and  $\mu$ on $\Y$ and $\z$, respectively.
  \end{enumerate}
  \label{remark:nonlinear}
\end{remark}

It follows from Theorem \ref{thm:nonlinear_invariance} that a fairly large family of distributions admit \Fequivalence{}. However, the exact representation of these families using the index $\phi$ is not entirely obvious. Nevertheless, the elements in $\Phi$ may be indexed by a more well-behaved set $\Theta$ such that the family $\pb{\ftheta : \theta \in \Theta}$ admits \Fequivalence{}. This is made precise in the following examples. The following example illustrates \Fequivalence{} when $\nu$ has a non-trivial modular character but $\X$ admits global cross sections.



\begin{example}
  Consider the spherical decomposition $\xv = (r, \thetav)$ of $\R^d \setminus \pb{\zerov}$ where $r \in \R_+$ and $\thetav \in \S^{d-1}$. Let $\mu$ be a measure on $\z = \R_+$ with density $g$ w.r.t. $\eta$ such that
  \eq{
    \eta(dr) = d(r^d) = r^{d-1}dr.
    \label{eq:nontrivial_modular_character}
  }
  For a nonnegative valued function $\xi: \mathbb{S}^{d-1} \rightarrow \R_+$, and for $\xv \in \X$ with spherical coordinates $(r, \thetav)$, let $f_\xi$ given by
  \eq{
  f_{\xi}(\xv) = {C_d} \cdot g\qty\big(r \cdot \xi(\thetav)),\label{eq:spherical}
  }
  where $1/C_d = 2\pi^{d/2} / \Gamma(d/2)$ is the surface area of $\S^{d-1}$. In the context of Theorem~\ref{thm:nonlinear_invariance}, consider the fiber bundle representation of $\X = \R^d \setminus \pb{\zerov}$ given by $\e = (\X, \Y, \z, \pi, G)$ where $\Y = \S^{d-1}$ denotes the typical fiber, the base space $\z = \R_+$, and the projection map  $\pi(\xv) = \norm{\xv}$. Taking $V=\z=\R_+$, and $\psi\inv: \X \rightarrow V \times \Y$ to be the map given by $$
    \psi\inv(\xv) = \qty(\norm{\xv}, \thetav) \qq{where} \theta_i = \arctan\qty\Big(\sqrt{x_{i+1}^2 + \dots + x_n^2} \big/  x_i), \qq{for} 1 \le i \le n-1,
  $$
  provides a local trivialization $\pb{(V, \psi)}$ for $\X$. Therefore, the density in \eref{eq:spherical} is equivalently given by
  \eq{
    f_\xi(\xv) = C_d \cdot g\qty( \phi( \psi\inv_r(\xv), r )),\nn
  }
  where $r = \pi(\xv)$ and $\phi(\thetav, r) = r \cdot \xi(\thetav)$. It is easy to verify that the map $\phi(\thetav, \cdot) \in \Delta(\z)$ for all $\thetav \in \S^{d-1}$, and $\phi\inv(\thetav, r) = r / \xi(\thetav)$. Importantly, from \eref{eq:nontrivial_modular_character}, the modular character of $\eta$ is $\Psi(t) = t^{d}$. If we consider the set $\Xi$, satisfying the constraint in \eref{eq:jac_constraint},
  \eq{
    \Xi = \qty{ \xi : \int_{\S^{d-1}} C_d \cdot \xi(\thetav)^{-d} \nu(d\thetav) = 1  },\nn
  }
  then from Theorem~\ref{thm:nonlinear_invariance} it follows that $\P(\Xi) = \pb{f_\xi: \xi \in \Xi}$ admits \Fequivalence{}. Moreover, from Remark~\ref{remark:nonlinear}~\ref{remark:measure-zero}, $f_\xi(\xv)$ also satisfies \Fequivalence{} on $\R^d$.
  \label{ex:spherical}
\end{example}

\begin{figure}
  \centering
  \includegraphics[width=0.9\linewidth]{\Root/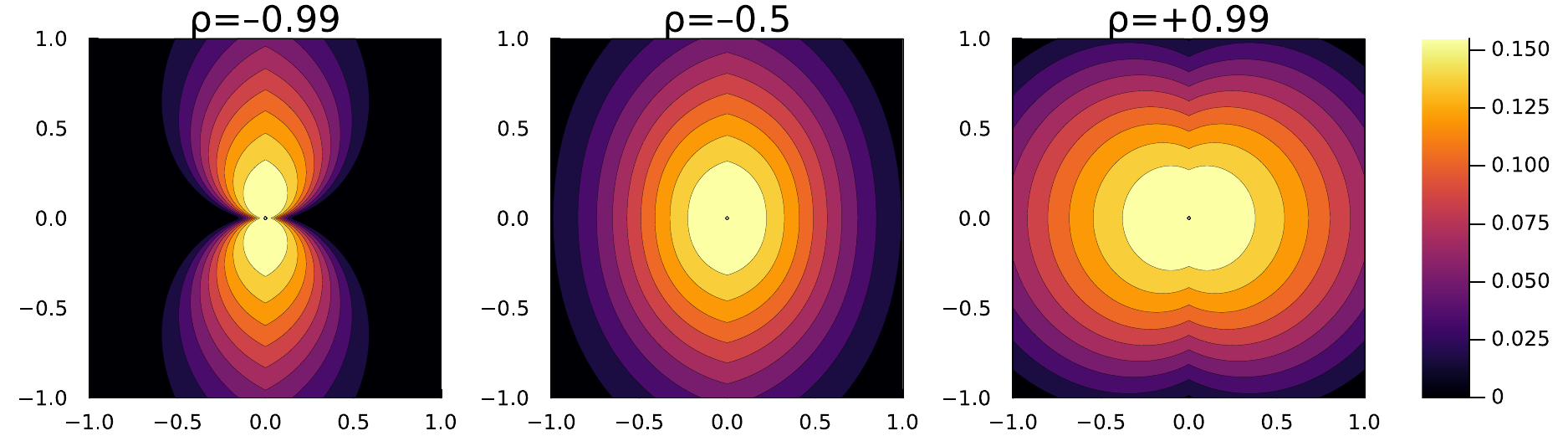}
  \caption{The sets $\qty\big{\xv \in \R^2 : f_\rho(\xv) \ge t}$ for $\rho \in \{-0.99, -0.5, 0.99\}$ respectively. For a fixed level $t$, all three of them have the same mass. In general, $\pr_{\rho}\qty\big({\{\xv \in \R^2 : f_\rho(\xv) \ge t}\})$ is the same for each $\abs{\rho}<1$.}

  \label{fig:invariance_radial}
\end{figure}

\begin{remark}

While seemingly contrived, the preceding example arises naturally in several situations. Note that the two base measures $\mu$ and $\nu$ on $\z = \R_+$ and $\Y = \S^{d-1}$, respectively, induce the standard Lebesgue measure $\lambda_d$ locally in $\R^d$, $\lambda_d = \eta \otimes_{\textup{loc}} \nu$, i.e., for $\xv = (r, \thetav)$ in local coordinates
$\lambda_d(d\xv) = \eta(dr) \cdot \nu(d\zv)$. In other words the non-trivial modular character in $\eta$ arises naturally when the standard Lebesgue measure on $\R^d$ is decomposed in spherical coordinates.

For a more concrete version of Example~\ref{ex:spherical}, let $d=2$ such that $\z = S^{1} \simeq [0, 2\pi]/\!\sim$ where $\qty{0} \sim \pb{2\pi}$, and for $\xv = (r\cost, r\sint)$ we have $\lambda_2(d\xv) = r \ dr \ d\theta$. Consider
$$
f(\xv) = \frac{g\qty\big(r \cdot \xi_\rho(\theta))}{2\pi},
$$
where $g(r)$ and $\xi_{\rho}(\theta)$ from Example~\ref{ex:spherical} are given by
\eq{
  g(r) = \pi\inv \cdot \exp({-{r^2}\big/{4\pi}}), 
  \quad \text{and }\quad 
  \xi_{\rho}(\theta) = 
  \pa{
    {1+\rho \cos\theta}
  }^{-1/2}.
  \label{eq:spherical}
  }
  Note that $\xi_\rho(\theta)$ is well-defined whenever $\abs{\rho}<1$, and
  $$
  \int_{0}^{2\pi} {}{\xi^{-2}_\rho(\theta)} \frac{d\theta}{2\pi} = 1.
  $$
  Therefore $\pb{f_\rho(\xv) : \abs{\rho} <1}$ admits \Fequivalence{}. The superlevel sets, $\pb{\xv \in\!\R^2\!: f_\rho(\xv)\!\ge t}$, of $f_\rho$ are shown in Figure~\ref{fig:invariance_radial}.
  
\end{remark}


The next example illustrates \Fequivalence{} for a family of distributions supported on the surface of a Möbius band with a full twist.

\begin{example}\label{ex:mobius}
  Let $\X$ be the surface of a Möbius band with a full twist given by
  \eq{
  \X = \qty\Big{ \xv(\theta, t) \in \R^3: \theta \in [0, 2\pi), \ t \in [-0.5, 0.5]},\nn
  }
  where
  \eq{
    \xv(\theta, t) = \qty\Big(\cos\theta(1 + t\cos\theta), \  \sin\theta(1 + t \cos\theta), \ t\sin\theta),\nn
  }
  as illustrated in Figure~\ref{fig:mobius-appendix-a}. For $\alpha \in \R$, let $\fa$ be a  density function on $\X$ given by
  \eq{
    \fa(\xv) = \fa\qty( \xv(\theta, t) ) = \frac{\exp\qty(\kappa \cdot \cos\qty(\alpha t + \theta))}{2\pi I_0(\kappa)},
    \label{eq:fa-mobius}
  }
  where $\kappa \in \R_+$ is a fixed parameter and $I_0(\kappa)$ is the modified Bessel function of the first kind. Then $\pb{\fa: \alpha \in \R}$ admits \Fequivalence{}. This follows by noting that $\X$ can be written as a fiber bundle with base space $\z = \S^1$ and typical fiber $\Y = [-0.5, 0.5]$. By identifying $\z \simeq [0, 2\pi]/\!\sim$ where $\qty{0} \sim \qty{2\pi}$, the projection map is given by $\pi(\xv) = \arctan({x_2/x_1})$. The collection $\pb{(V_j, \psi_j): j = 1, 2}$ provide the local trivialization for $\X$, where
  $$
    V_1 = (0, 2\pi) \setminus \pb{\pi}, \qq{and} V_2 = [0, \epsilon) \cup (\pi - \epsilon, \pi + \epsilon)
  $$
  for sufficiently small $0 < \epsilon < \pi/2$, and for $z = \pi(\xv) \in [0, 2\pi)$
  \eq{
    \psi\inv_{1, z}(\xv) = \frac{x_3}{\sin z}, \qq{and} \psi\inv_{2, z}(\xv) = \frac{1}{\cos z}\qty(\frac{x_1}{\cos z} - 1).\nn
  }
  A more detailed description is provided in Appendix~\ref{ex:mobius-appendix}. The density in \eref{eq:fa-mobius} may be rewritten as
  $$
    \fa(\xv) = g\qty( \phi_\alpha(\psi\inv_{j, \theta}(\xv), \theta) ),
  $$
  where $g(\theta) = \exp({\kappa \cos\theta}) / 2\pi I_0(\kappa)$ is the density of the von Mises distribution on $[0, 2\pi]/\!\sim$ with parameter $\kappa$, and modular character $\Psi(u) = u$. The map $\phi_{\alpha}(\theta, t) = \alpha t + \theta \mod 2\pi$ is such that $\phi_\alpha(\cdot, t) \in \Delta(\z)$ for all $t \in [-0.5, 0.5]$. The Jacobian constraint in \eref{eq:jac_constraint},
  \eq{
    \int_{-0.5}^{0.5} \absdetj{\phi\inv_\alpha(\cdot, t)} d\nu(t) = \int_{-0.5}^{0.5} \abs{\frac{d}{d\theta} (\theta - \alpha t \mod 2\pi)} d\nu(t) = \int_{-0.5}^{0.5} 1 \cdot dt = 1,\nn
  }
  is satisfied for all $\alpha \in \R$. Therefore, by Theorem~\ref{thm:nonlinear_invariance}, $\pb{\fa: \alpha \in \R}$ admits \Fequivalence{}. Figure~\ref{fig:mobius} shows the superlevel sets $\pb{\xv \in \X : \fa(\xv) \ge t}$ when $\alpha \in \pb{0, -5, 8}$.
\end{example}

\begin{figure}
  \centering
  \begin{subfigure}[t]{0.97\linewidth}
    \includegraphics[width=0.32\linewidth]{\Root/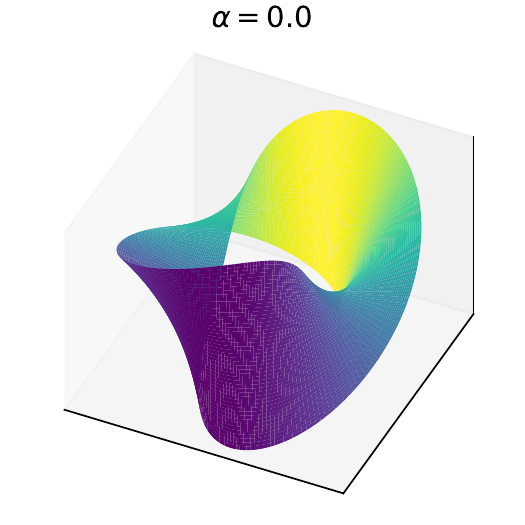}
    \includegraphics[width=0.32\linewidth]{\Root/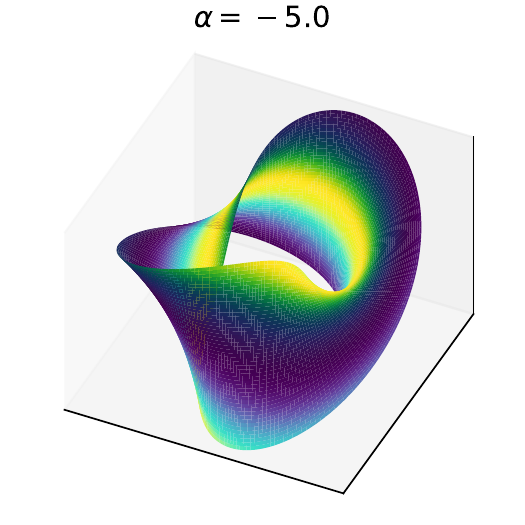}
    \includegraphics[width=0.32\linewidth]{\Root/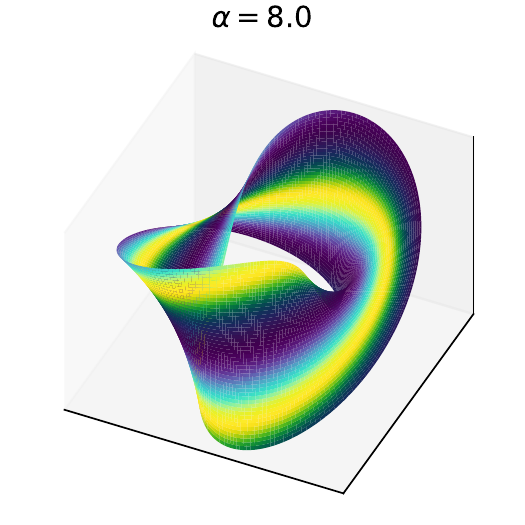}
  \end{subfigure}
  \hspace*{-1em}
  \begin{subfigure}[t]{0.01\linewidth}
    \includegraphics[width=4\linewidth]{\Root/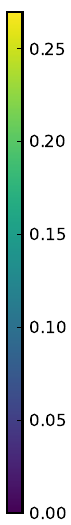}
  \end{subfigure}
  \caption{Superlevel sets of the probability density function $\fa$ on $\X$ from Example~\ref{ex:mobius} when $\alpha \in \pb{0, -5, 8}$. For a fixed level $t \ge 0$, the mass of the superlevel set $\pb{\xv \in \X : \fa(\xv) \ge t}$ is the same for each $\alpha$.}
  \label{fig:mobius}
\end{figure}

For the fiber bundle representation $\e = \pa{\X, \Y, \z, \pi, G}$, when the typical fiber space $\Y$ is discrete, Theorem~\ref{thm:nonlinear_invariance} simplifies to the following result.

\begin{corollary}\label{cor:discrete}
  Let $\X \subset \R^d$, $\mu$ be a measure on $\X$, and $\pi: \X \rightarrow A$ be a given continuous surjection from $\X$ to $A \subset \X$. Suppose $\X$ admits the representation $\X = \qty(\sqcup_{i=1}^N A_i) \bigsqcup \qty(\cup_{j=1}^m B_i)$ where each $A_i$ is diffeomorphic to $A$ and $\mu(B_j) = 0$ for each $j$. For a collection of maps $\phiv \defeq \pb{\phi_1, \phi_2, \dots, \phi_N}$ such that each $\phi_i \in \Delta(A)$, let $\fphiv$ be a probability density function on $\X$ given by
  \eq{\label{eq:discrete-density}
    \fphiv(\xv) = \frac{1}{N} \sum_{i=1}^N g\qty\big(\phi_i(\pi(\xv))) \cdot \mathbbm{1}(\xv \in A_i),
  }
  where $g$ is a probability density function with $\emph{supp}(g) = A$. Then, $\pb{f_{\phiv} : \phiv \in \Phiv_N}$ admits \Fequivalence{} for
  \eq{
  \Phiv_N \defeq \pb{\phiv = \pb{\phi_1,\phi_2,\dots,\phi_N} : \frac{1}{N} \sum\limits_{i=1}^N {\Psi_\mu\pa{\absdetj{\phi\inv_i}} = 1}}.\nonumber
  }
\end{corollary}

\begin{remark}
  For $t \in \R$, the excess mass function for $\fphiv$ is given by
  $
    \hat{f}_{\phiv}(t) = \hat{g}(Nt). 
  $
  Therefore, $N \in \Z_+$ is a fixed value in the family of distributions $\qty\big{\fphiv: \phiv \in \Phiv_N}$. In other words, if $f_{\phiv_1}, f_{\phiv_2}$ admit the representation in \eref{eq:discrete-density} and $f_{\phiv_1} \Feq f_{\phiv_2}$, then $|{\phiv_1}| = |{\phiv_2}|$.
\end{remark}

For completeness, the proof is provided in Section~\ref{proof:cor:discrete}. The following two examples illustrate \Fequivalence{} for distributions which admit the representation in Corollary~\ref{cor:discrete}.


\begin{figure}
  \centering
  \begin{subfigure}[b]{0.32\linewidth}
    \includegraphics[width=\linewidth]{\Root/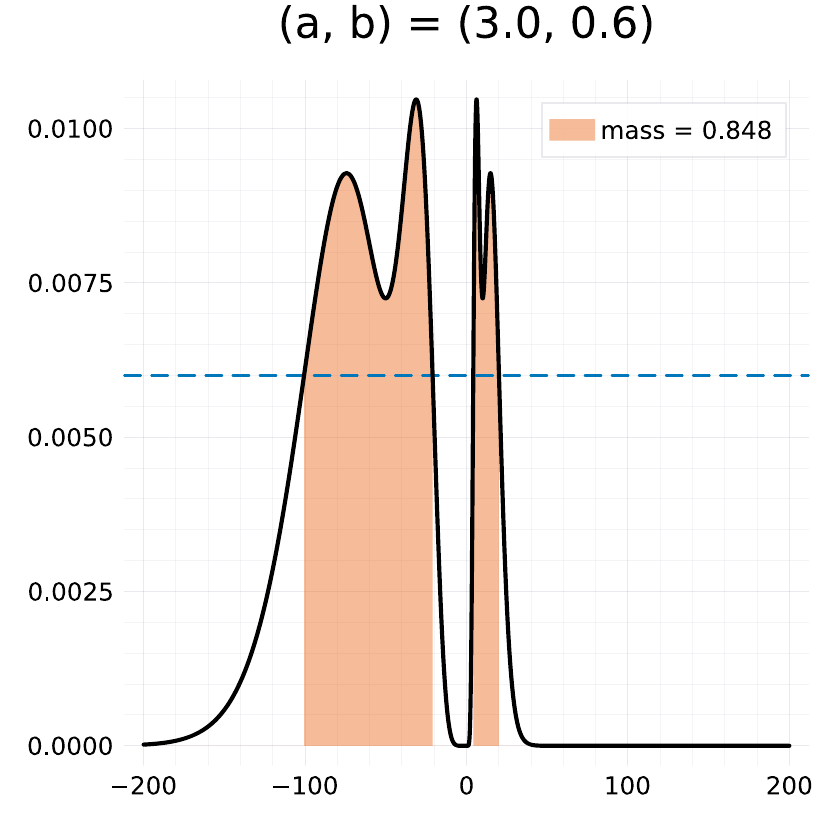}
  \end{subfigure}
  \begin{subfigure}[b]{0.32\linewidth}
    \includegraphics[width=\linewidth]{\Root/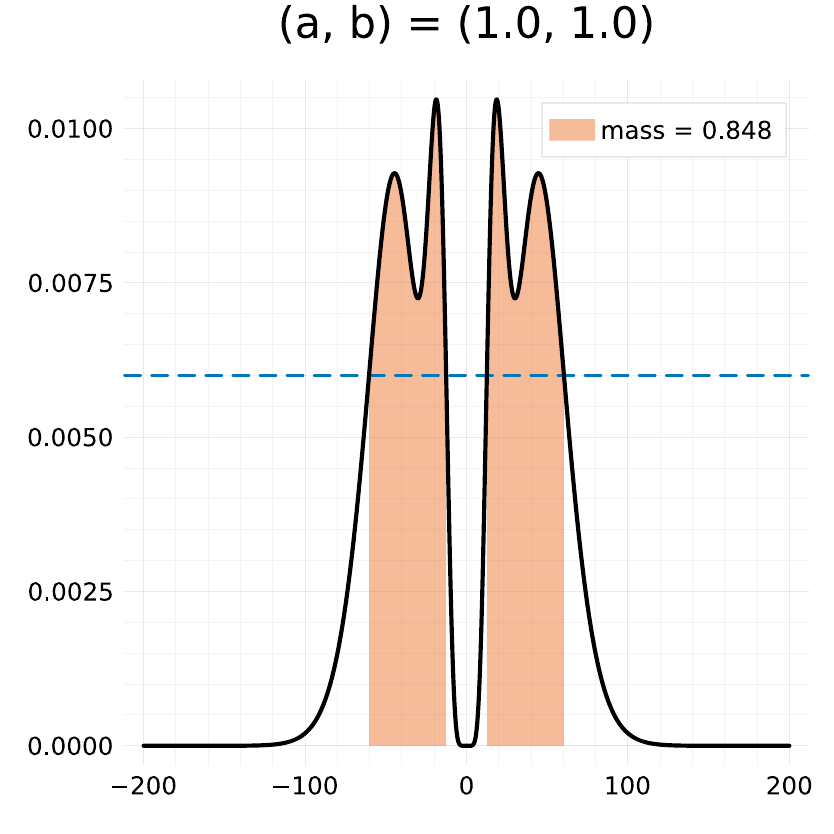}
  \end{subfigure}
  \begin{subfigure}[b]{0.32\linewidth}
    \includegraphics[width=\linewidth]{\Root/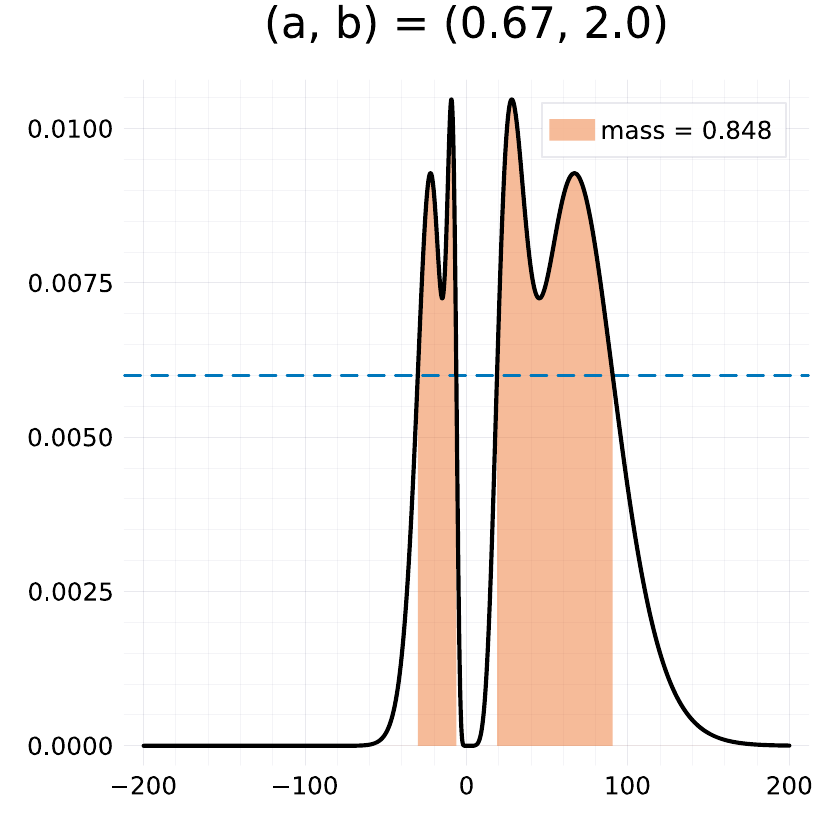}
  \end{subfigure}
  \caption{Illustration of \Fequivalence{} when $g \sim 0.5 \Gamma(10,5) + 0.5 \Gamma(10,2)$ is a mixture of Gamma distributions. For different values of $\theta \in \Theta$, the excess mass functions $\hat{f}_{\theta}(t)$ are identical, as shaded in orange for $t=0.006$.}
  \label{fig:invariance_univariate}
\end{figure}

\begin{example}\label{ex:mass_transport}
  For $\thetav = (\ta_1, \ta_2)$, let $\Theta \defeq \pb{(\ta_1, \ta_2) \in \Rpp : 1/\ta_1 + 1/\ta_2 = 2}$, and consider the family of distributions $\pb{\fthetav: \thetav \in \Theta}$ on $\X = \R$ given by
  \eq{
    \fthetav(x) = \begin{cases}
      g(\theta_1 x) / 2,  & \text{if } x > 0 \\
      g(-\theta_2 x) / 2, & \text{if } x \le 0
    \end{cases},
    \label{eq:mass_transport}
  }
  where $g$ is a probability density function on $\R_+$. Then, $\pb{\ftheta: \theta \in \Theta}$ admits \Fequivalence{}. This can be seen by representing $\X = \Rp \sqcup \R_- \sqcup \pb{0}$, where $A_1 = \R_+$ and $A_2 = \R_-$ are both diffeomorphic to $A = \R_+$, and $\pb{0}$ has measure zero. By taking $\pi(\xv) = \abs{x}$ to be the continuous surjection from $\X$ to $A$, and the maps $\phi_1, \phi_2 \in \Delta(\R_+)$ to be
  $$
  \phi_1(z) = \theta_1 z \qq{and} \phi_2(z) = \theta_2 z,
  $$
  the density function, $\fthetav$, is equivalently expressed as
  \eq{
    \fthetav(x) = \frac{1}{2} \sum_{i=1}^2 g\qty\Big( \phi_i(\abs{x} )) \cdot \mathbbm{1}(x \in A_i). \nn
  }
  Since $\phiv = \pb{\phi_1, \phi_2}$ is completely specified by $\thetav = (\theta_1, \theta_2)$, from Corollary~\ref{cor:discrete}, the Jacobian constraint is satisfied when
  \eq{
    \sum_{i=1}^2 \half \cdot \Psi_\mu\pa{\absdetj{\phi_i\inv}} = \frac{1}{2 \theta_1} + \frac{1}{2 \theta_2} = 1.\nn
  }
  Therefore, $\pb{\fthetav: \thetav \in \Theta}$ admits \Fequivalence{} from Corollary~\ref{cor:discrete}. See Figure~\ref{fig:invariance_univariate} for an illustration when $g$ is the density function associated with the mixture of Gamma distributions: $0.5 \cdot \Gamma(10, 5) + 0.5 \cdot \Gamma(10, 2)$.
\end{example}

\begin{example}\label{ex:mass_transport_2}
  For $\thetav = (\theta_1, \theta_2, \theta_3, \theta_4) \in \R_{++}^4$ consider the probability density $\fthetav$ on $\X = \R^2$ from \eref{eq:biv-chi} in Example~\ref{ex:biv-chi}. To see that $\pb{\fthetav: \theta_1 + \theta_2 + \theta_3 + \theta_4 = 4}$ admits \Fequivalence{}, consider the representation of $\X = \R^2$ as
  \eq{
    \X = \qty{\bigsqcup_{i=1}^4 A_{i}}  \bigsqcup \ \ \underbrace{\qty\Big{(x_1, x_2): x_1=0 \text{ or } x_2=0}}_{=: B},
  }
  where $B$ is a set of measure zero, and $A_1, A_2, A_3, A_4$ are the four (open) quadrants of~$\R^2$ which are each diffeomorphic to $A = \Rpp = \pb{(z_1, z_2) \in \R^2: z_1, z_2 > 0}$. 
  By taking $\pi(x_1, x_2) = (\abs{x_1}, \abs{x_2})$ to be the continuous surjection from $\X$ to $A$, consider the maps $\phi_i \in \Delta(\Rpp)$ to be given by $\phi_i(\zv) = R(\theta_i)\tr\zv$ where 
  $$
  R(\theta_i) = \begin{pmatrix}
    1/\sqrt{\theta_i} & 0               \\
    0               & 1/\sqrt{\theta_i} \\
  \end{pmatrix}.
  $$
 Then, $\fthetav$ is equivalently expressed as
  \eq{
    \fthetav(\xv) = \frac{1}{4} \sum_{i=1}^4 g\qty\Big( \phi_i(\pi(\xv))) \cdot \mathbbm{1}(\xv \in A_i).\nn
  }
  Similar to Example~\ref{ex:mass_transport}, $\phiv = \pb{\phi_1, \phi_2, \phi_3, \phi_4}$ is completely specified by $\thetav \in \R^4$, and the Jacobian constraint in \eref{eq:jac_constraint} is satisfied when
  \eq{
    \sum_{i=1}^4 \f 14 \absdetj{\phi\inv_i} = \sum_{i=1}^4 \frac{1}{4} \cdot \det\qty({R(\theta_i)\inv}) = \frac{1}{4} \cdot \sum_{i=1}^4 \theta_i = 1,\nn
  }
  and, therefore, the family $\pb{\fthetav: \theta \in \Theta}$ admits \Fequivalence{}. 
  \end{example}
}


While Theorem~\ref{thm:group_invariance} and Theorem~\ref{thm:nonlinear_invariance} provide conditions for $\P(\Theta)$ to admit \Fequivalence{} when $\ftheta$ has a specific form, we might ask: \emph{What happens to distributions that do not conform to the templates mentioned above}? The next result provides a necessary and sufficient geometric constraint (in the space of probability distributions) which $\P(\Theta)$ needs to satisfy for \Fequivalence{}.

\begin{theorem}
  For an open set $\Theta \subseteq \R^p$, let $\P(\Theta) = \pb{\ftheta : \thetav \in \Theta}$ be a family of distributions on $\X$. Then $\P(\Theta)$ admits \Fequivalence{} if and only if for all $1\le i\le p$,
  \eq{
    \f{\partial}{\partial \theta_i} \pa{ \int_{\X}{\fthetav^{k+1}(\xv)d\xv}} = 0,\ \text{for all} \ k \in \mathbb N_0.
    \label{eq:deriv-condition}
  }
  Moreover, if the gradient $\nabla_{\thetav}\fthetav$ exists a.e.-$\lambda_d$ and there exists a function $M \in L^1\pa{\X}$ such that for each $1 \le i \le p$ and for all $\theta \in \Theta$, $\abs{\f{\partial}{\partial \theta_i}\fthetav(\xv)} \le M(\xv)$ a.e.-$\lambda_d$, then $\P(\Theta)$ admits \Fequivalence{} if and only if
  \eq{
    \Big< \fthetav^k, \nabla_{\thetav} \fthetav \Big> _{L^2(\X)} = \zerov,\quad \text{for all } k \in \mathbb N_0.
    \label{eq:gradient}
  }
  \label{thm:orthogonal}
\end{theorem}

We collect the proof in Section~\ref{proof:thm:orthogonal}, and illustrate Theorem~\ref{thm:orthogonal} by verifying \Fequivalence{} for the family of distributions already studied in Example~\ref{ex:mass_transport}.


\begin{example}
  Consider the family of distributions on $\R^2$ from Example \ref{ex:spherical} in \eref{eq:spherical} given by,
  \eq{
    f_\rho(\xv) = g(r\cdot\xi_{\rho}(\theta)) = \f{1}{2\pi}\exp\pa{\f{-r^2}{2(1+\rho\cos(\theta))}},\nn
  }
  where $r = \norm{\xv}$ and $\tan(\theta) = x_2/x_1$. First, observe that $\rho \in [-1,1]$ contains an open set in~$\R$. Moreover, $f_\rho$ clearly satisfies the stochastic regularity assumptions in Theorem \ref{thm:orthogonal} by taking
  $$
    M(\xv) = \frac{r^2|\cos(\theta)|}{4\pi(1-|\cos(\theta)|)}\exp(-{r^2}/{4}).
  $$
  The derivative of $f_\rho$ is given by
  $$
    \f{\partial}{\partial \rho}f_\rho(\xv) = \f{r^2\cos(\theta)}{2(1+\rho\cos(\theta))^2} \cdot \f{1}{2\pi} \exp\pa{\f{-r^2}{2(1+\rho\cos(\theta))}}.
  $$
  Then, for any $k \ge 0$, \eref{eq:gradient} now becomes
  \eq{
  \Big< f_\rho^k, \f{\partial}{\partial \rho}f_\rho \Big> _{L^2(\R^2)} &=  \int\limits_{0}^{\infty}\int\limits_{0}^{2\pi}{  \f{r^2\cos(\theta)}{2(1+\rho\cos(\theta))^2} \cdot \f{1}{(2\pi)^{k+1}} \cdot \exp\pa{-\f{(k+1)r^2}{2(1+\rho\cos(\theta))}} \cdot rdrd\theta }\nn\\
  &\stackrel{\textup{(i)}}{=} \int_{0}^{2\pi}{\f{\cos(\theta) d\theta}{(k+1)^2\cdot(2\pi)^{k+1}}} = 0,\nn
  }
  where (i) follows from making the substitution $t = r^2$. It follows from Theorem~\ref{thm:orthogonal} that $\pb{f_\rho: \abs{\rho}<1}$ admits \Fequivalence{}.
\end{example}


\begin{example}
  For $\Theta = (1,\infty)$, consider the family of distributions on $\R$ given by
  \eq{
    \ftheta(x) = \half \ g(\theta x)\mathbbm{1}\pa{x \ge 0} + \half \ g\pa{\f{\theta x}{1-2\theta}}\mathbbm{1}\pa{x < 0},
    \label{eq:mass_trans2}
  }
  where $g$ is any density on $\R_+$ satisfying the assumptions of Theorem \ref{thm:orthogonal}. Then, $\pb{\ftheta : \theta \in \Theta}$ admits \Fequivalence{}. Vis-\`a-vis Example~\ref{ex:mass_transport}, the density in \eref{eq:mass_trans2} is a {reparametrization} of the density in \eref{eq:mass_transport} to ensure that $\Theta \subset \R$ is an open set. Observe that
  $$
    \int\limits_{\R}{\ftheta(x)dx} = \f{1}{2\theta} + \f{2\theta-1}{2\theta} = 1,
  $$
  implying that $\ftheta$ is a well defined density function for all $\theta > 1$. In order to verify the condition in \eref{eq:deriv-condition} note that
  {
  \eq{
  \f{\partial}{\partial\theta} \int_{\R}\ftheta^{k+1}(x) dx 
  &= \frac{1}{2^{k+1}}\ \cdot \f{\partial}{\partial \theta} \pa{\int_{\R}  g^{k+1}(\theta x)\mathbbm{1}\pa{x \ge 0}dx + \int_{\R} \ g^{k+1}\pa{\f{\theta x}{1-2\theta}} \mathbbm{1}\pa{x < 0}dx}\nn\\
  &\stackrel{\text{(i)}}{=} \frac{1}{2^{k+1}}\ \cdot \f{\partial}{\partial \theta} \pa{\pa{\f{1}{\theta} - \f{1-2\theta}{\theta}} \cdot \int\limits_{0}^{\infty}{g^{k+1}(t)dt}} = \f{\partial}{\partial\theta}\int\limits_{0}^{\infty}{\frac{1}{2^k} \cdot g^{k+1}(t)dt} = 0,\nn
  }
  }
  where (i) follows from taking $t=\theta x$ in the first integral and $t = {\theta x}/{1-2\theta}$ in the second integral.
  By Theorem \ref{thm:orthogonal}, this implies that $\pb{\ftheta : \theta \in \Theta}$ admits \Fequivalence{}.
  \label{ex:mass_trans2}
\end{example}


\section{Proofs}
\label{sec:proofs}
\allowdisplaybreaks

In this section, we present the proofs for the main results of this paper.


\subsection{Proof of Lemma~\ref{lem:inclusion}}
\label{proof:lem:inclusion}

Consider $\pr,\qr \in \P$ with their respective probability density functions $f$ and $g$,and consider $\Xv \sim f$ and $\Yv \sim g$. For every fixed function $\gamma_k$, the equivalence $f \Feq g$ implies that
$$
  \gamma_k(s f(\Xv)^{1/d}, t f(\Xv)^{1/d}) \distas \gamma_k(s g(\Yv)^{1/d}, t g(\Yv)^{1/d}), \qq{for all} 0 \le s < t.
$$
It follows that $\mu_k(\pr; s, t) = \mu_k(\qr; s, t)$ for all $0 \le s < t$. A similar argument also holds for $\varsigma_k$. This implies that $\beta_k(\pr;s, t) \distas \beta_k(\qr;s, t)$. Since this holds for each $k \ge 0$ we have that $\pr \beq \qr$, and the result follows.\QED


\subsection{Proof of Theorem~\ref{thm:group_invariance}}
\label{proof:thm:group_invariance}

\textit{Part (i).} Suppose $\Xv_\theta$ is a random variable with density $\ftheta$. Consider the random variable $\ytheta \defeq \gtheta\pa{\Psi\pa{\xtheta}}$ as a transformation of $\xtheta$, such that $\ftheta(\xtheta) = \phi(\ytheta)$. By Lemma~\ref{lem:inclusion}, if we can show that the distribution of $\ytheta$ does not depend on the parameter $\theta$ if and only if $\det\pa{\mathbf{D}{\Psi\inv}\pa{\yv}} = \zeta\pa{T\pa{\yv}}$ for some function $\zeta: \Tau \rightarrow \R$, then \Fequivalence{} for the family of distributions $\P(\Theta)$ follows.

The inverse transformation for $\ytheta$ is given by $\xv = \Psi\inv \circ \gtheta\inv(\yv)$. The existence of $\gtheta\inv$ is guaranteed by the group $\G$. The Jacobian for the inverse transformation can be simplified using the multivariable chain-rule,
\eq{
  \mathbf{D}\pa{\Psi\inv \circ \gtheta\inv}\pa{\yv} = \mathbf{D}{\Psi\inv}\pa{\gtheta\inv(\yv)} \cdot \mathbf{D}{\gtheta\inv}\pa{\yv}.\nn
}
Since $\G$ is a group of isometries, we have that, $\abs{\det\pa{\mathbf{D}{\gtheta\inv}\pa{\yv}}}=1$. The density of $\ytheta$ is expressed as
\eq{
  h_\theta(\yv) = \phi(\yv) \cdot \norm{ \mathbf{D}{\Psi\inv}\pa{\gtheta\inv(\yv)}}.\nn
}

It follows that density $h_\theta$ does not depend on $\theta$ if and only if {$\det\pa{\mathbf{D}{\Psi\inv}\pa{\gtheta\inv(\yv)}}$} does not depend on $\theta$, i.e., $\det\pa{\mathbf{D}{\Psi\inv}(\yv)}$ is $\G$-invariant. By Proposition \ref{lem:maximal_invariant}, this holds if and only if there exists some function $\zeta : \Tau \rightarrow \R$ such that
\eq{\label{eq:det-constraint}
  \det\pa{\mathbf{D}{\Psi\inv}(\yv)} = \zeta\pa{T(\yv)},\nn
}
where $T$ is $\G$-maximal invariant. Therefore, the distribution of $\ytheta$ doesn't depend on $\thetav$ if and only if the condition in \eref{eq:det-constraint} holds. Since $\phi$ is a fixed function and $\ftheta(\Xv_\theta) = \phi(\ytheta)$, this implies that \eref{eq:det-constraint} is also a necessary and sufficient condition for the distribution of $\ftheta(\Xv_\theta)$ to not depend on $\theta$. This concludes the proof for the first part of the Theorem~\ref{thm:group_invariance}.

\textit{Part (ii).} For ease of notation, let $\G = \mathop{\times}_{i=1}^{m}{\G_i}$ act on the space $\Y$. If $T_i: \Y_{i-1} \rightarrow \Y_{i}$ is a sequence of $\G_i$-compatible maximal invariants, then, for the second claim, from part (i), it suffices to show that
$$
  T(\yv) = T_m \circ T_{m-1} \circ \dots T_1(\yv)
$$
is $\G$-maximal invariant. The proof follows from induction. For the case $m=1$, $T(\yv) = T_1(\yv)$ by definition, so the property holds trivially. Assume that the property holds for $m > 1$. Then, $T' = T_m \circ T_{m-1} \circ \dots \circ T_1$ is $\G'$-maximal invariant, where $\G' = \times_{i=1}^{m}{\G_i}$.

Let $\G_{m+1}$ be a group acting on $\Y_m$ such that $T_{m+1} : \Y_m \rightarrow \Y_{m+1}$ is $\G_{m+1}$-maximal invariant. From the assumption that $T_{m+1}$ is $\G_{m+1}$-compatible, we also have that $T'$ is $\G_{m+1}$ compatible; therefore, we only need to show that $T = T_{m+1} \circ T'$ is $\G$-maximal invariant, where $\G = \G' \times \G_{m+1}$. Each element $g \in \G$ is given by $g = \pa{g',g_{m+1}}$ where $g' \in \G'$ and $g_{m+1} \in \G_{m+1}$. We can write $g$ as
\eq{
g = \pa{g',g_{m+1}} = \pa{g',e_{m+1}}*\pa{e',g_{m+1}} = \Tilde{g}'*\Tilde{g}_{m+1},\nn
}
where  $e'$ and $e_{m+1}$ are the identity elements of $\G'$ and $\G_{m+1}$ respectively, and $\Tilde{g}'$ and $\Tilde{g}_{m+1}$ are the representations for the group action of $\G'$ and $\G_{m+1}$ on $\Y$ as subgroups of $\G$. First, we examine that $T$ is $\G$-invariant. For each $\yv\in \Y$ we have that
\eq{
  T(g\yv) = T_{m+1}\circ T'(g \yv) = T_{m+1}\circ T'(\Tilde{g}'*\Tilde{g}_{m+1} \yv).
  \label{eq:product_groups1_appendix}
}
By the definition of the group action, we can write $\Tilde{g}'*\Tilde{g}_{m+1} \yv= \Tilde{g}'\zv$, where $\zv = \Tilde{g}_{m+1} \yv$. Since $T'$ is $\G'$-maximal invariant, we have that $T'(\Tilde{g}'\zv) = T'(\zv) = T'(\Tilde{g}_{m+1} \yv)$. Additionally, since $T'$ is $\G_{m+1}$-compatible, it follows that $T'(\Tilde{g}_{m+1} \yv) = g^*_{m+1} T'(\yv)$, where $g^*_{m+1}$ is the induced action of $\Tilde{g}_{m+1}$ on $\Y_m$ via $T'$. Lastly, using the fact that $T_{m+1}$ is $\G_{m+1}$-maximal invariant, \eref{eq:product_groups1_appendix} becomes
\eq{
  T(g\yv) = T_{m+1}\pa{g^*_{m+1} T'(\yv)} = T_{m+1} \circ T'(\yv) = T(\yv).\nn
}
Next, let $\xv$ and $\yv$ be such that $T(\xv) = T_{m+1} \circ T'(\xv) = T_{m+1} \circ T'(\yv) = T(\yv)$. Since $T_{m+1}$ is maximally invariant, there exists $g_{m+1} \in \G_{m+1}$ such that $g^*_{m+1}T'(\xv) = T'(\yv)$. From the $\G_{m+1}$-compatibility of $T'$ we have that $g^*_{m+1}T'(\xv) = T'(\Tilde{g}_{m+1}\xv)$, giving us $T'(\yv) = T'(\Tilde{g}_{m+1}\xv)$. Lastly, since $T'$ is $\G'$-maximal invariant, there exists $\Tilde{g}'$ such that $\Tilde{g}'\pa{\Tilde{g}_{m+1}\xv} = \yv$. This implies that there exists $g \in \G$ such that
\eq{
  g\xv= \Tilde{g}'*\Tilde{g}_{m+1}\xv= \Tilde{g}'\pa{\Tilde{g}_{m+1}\xv} = \yv.\nn
}
Therefore, $T(\xv) = T(\yv)$ if and only if $\xv \in \G\yv$, from which it follows that $T = T_{m+1} \circ T'$ is $\G$-maximal invariant.\QED

{


\allowdisplaybreaks
\subsection{Proof of Lemma~\ref{lem:excess_mass}}
\label{proof:lem:excess_mass}

Consider $\Xv \sim f$ and $\Yv \sim g$ for $f, g \in \P$, and let $Z_{\Xv} = f(\Xv)$ and $Z_{\Yv} = g(\Yv)$ be the transformation of $\Xv$ and $\Yv$ under their own density. Then, for $t \ge 0$
\eq{
\hat{f}(t) = \pr( f(\Xv) \ge t ) = \E( \mathbbm{1}(f(\Xv) > t) ) = \E(\mathbbm{1}(Z_{\Xv} > t)) = 1 - F_{Z_{\Xv}}(t),\nn
}
where $F_{Z_{\Xv}}(t)$ is the cumulative distribution function of $Z_{\Xv}$. Similarly, $\hat{g}(t) = 1 - F_{Z_{\Yv}}(t)$. 

For the first claim, note that if $Z_{\Xv} \distas Z_{\Yv}$, then $F_{Z_{\Xv}}(t) = F_{Z_{\Yv}}(t)$ for all $t \ge 0$, which implies that $\hat f = \hat g$. Conversely, if $\hat f = \hat g$, then $F_{Z_{\Xv}} = F_{Z_{\Yv}}$, which implies that $Z_{\Xv} \distas Z_{\Yv}$. Therefore, $f \Feq g$ if and only if $\hat f = \hat g$. The second claim now follows from Lemma~\ref{lem:inclusion}. \QED


\subsection{Proof of Theorem~\ref{thm:nonlinear_invariance}}
\label{proof:thm:nonlinear_invariance}

Let $\lambda = \nu \otimes\loc \mu$ be the local product measure induced in $\X$, and let $C = 1/\nu(\Y)$. First, we verify that $f_\phi$ as defined in \eref{eq:fiber} is a well-defined probability density function for each $\phi \in \Phi$. From \citet[Eq.~6]{goetz1959measures}

\begin{align*}
  \int_{\X} f_\phi d\lambda
  &=  \int_{\z} \int_{\pi\inv(\zv)}^{} C \cdot g\qty\Big( \phi\qty\big( \psi\jz\inv(\wv), \zv ) ) d\nuz(\wv) d\mu(\zv)\nn\\
  &\stackrel{\text{(i)}}{=}  \int_{\z} \int_{\psi\jz\inv \circ \pi\inv(\zv)}^{}  C \cdot g\qty( \phi\qty\big( \yv, \zv ) ) d\nuz\qty\big(\psi\jz(\yv)) d\mu(\zv)\nn\\
  &\stackrel{\text{(ii)}}{=}  \int_{\z} \int_{\Y}^{}  C \cdot g\qty(\phi_\yv(\zv)) d\nu\qty(\yv) d\mu(\zv)\nn\\
  &\stackrel{\text{(iii)}}{=}  \int_{\z} \int_{\Y}^{} C \cdot g\qty(\uv) d\nu\qty(\yv) d\mu\qty(\phi_\yv\inv(\uv))\nn\\
  &\stackrel{\text{(iv)}}{=} \int_{\z} \int_{\Y}^{} C \cdot g\qty(\uv) d\nu\qty(\yv) \Psi\qty(\absdetj{\phi_\yv\inv}) d\mu\qty(\uv)\nn\\
  &\stackrel{\text{(v)}}{=} \int_{\Y}^{}  C \cdot \Psi\qty(\absdetj{\phi_\yv\inv}) \qty{\int_{\z}  g\qty(\uv) d\mu\qty(\uv) } d\nu\qty(\yv) \nn\\
  &= \int_{\Y}^{}  C \cdot \Psi\qty(\absdetj{\phi_\yv\inv}) d\nu\qty(\yv) = 1,  \nn\\
\end{align*}

where (i) follows from making the substitution $\yv = \psi\jz\inv(\wv)$, (ii) follows from noting that $\psi\jz\inv\qty(\pi\inv(\zv)) = \Y$ by definition of the local trivialization and the pushforward measure $\nuz$ is defined to be $\nuz(B) = \nu(\psi\jz\inv(B))$ for all measurable sets $B \subseteq \pi\inv(\zv)$. Similarly, (iii) follows by making the substitution $\uv = \phi_\yv(\zv)$, (iv) follows from the fact that $\phi_\yv \in \Delta(\z)$ for every $\yv \in \Y$ and the modular character of $\mu$ is $\Psi$, and (v) follows from Tonelli's theorem \citep[Theorem~2.7]{folland1999}. Next, it remains to verify that $\fhat_\phi(t)$ does not depend on $\phi$, and we use the same machinery as before. Consider,

\eq{
  \fhat_\phi(t)
  &= \int_{\X}{\mathbbm{1}(f_{\phi}(\xv) \ge t) \ f_{\phi}(\xv) d\lambda(\xv)}\nn\\
    &= \int_{\z}\int_{\pi\inv(\zv)}{\mathbbm{1}\qty\bigg({ C \cdot }{g\qty\Big( \phi( \psi\jz\inv(\wv), \zv ) )} \ge t) \ {\ C \cdot \ } g\qty\Big( \phi( \psi\jz\inv(\wv), \zv ) ) \cdot d\nuz(\wv)} d\mu(\zv).\nn
  }

  Again, substituting $\yv = \psi\jz\inv(\wv)$, and then taking $\uv = \phi_\yv(\zv)$ we get
  \eq{
    \fhat_\phi(t)
    &= \int_{\z}\int_{\Y}{\mathbbm{1}\qty({ C \cdot } {g(\uv)} \ge t) \ { C \cdot } g(\uv) \ \Psi(\absdetj{\phi_\yv\inv}) \cdot d\nu(\yv)} d\mu(\uv)\nn\\
    &= \int_{\z}{\mathbbm{1}\qty({ C \cdot } {g(\uv)} \ge t) \ g(\uv) \ \qty{\int_\Y {\ C \cdot \ } \Psi(\absdetj{\phi_\yv\inv}) \cdot d\nu(\yv)}} d\mu(\uv)\nn\\
    &= \int_{\z}{\mathbbm{1}\qty({ C \cdot } {g(\uv)} \ge t) \ g(\uv)} d\mu(\uv) = \hat{g}(t/C),\nn
  }
  which does not depend on the choice of $\phi \in \Phi$, and \Fequivalence{} follows.\QED

  \subsection{Proof of Proposition~\ref{prop:coordinate-equivalence}}
  \label{proof:prop:coordinate-equivalence}

  For a fixed $\phi: \Y \times \z \rightarrow \z$, let $\fphi$ and $\ftphi$ be the density functions given by \eref{eq:fiber} and \eref{eq:fiber-alt}:
  \eq{
    f_\phi(\xv) = {\ C \cdot \ }g\qty\Big( \phi\qty\big( \psi\jz\inv(\xv), \zv ) ), \quad \text{and} \quad \ftphi(\xv) = {\ C \cdot \ }g\qty\Big( \phi\qty\big( \eta\iz\inv(\xv), \zv ) ),\nn
  }
  corresponding to the local trivializations $\pb{(V_j,\psi_j) : j \in J}$ and $\pb{(U_i,\eta_i) : i \in I}$, respectively. Without loss of generality, we may take $C=1$ for ease of notation. For a measurable set $A \subset \X$, and for $i \in I$ and $j \in J$ such that $\zv \in \pi(A) \cap U_i \cap V_j$, using \eref{eq:fubini} we get
  \eq{
    \int_{A}\fphi d\lambda = \int_{\pi(A)} \int_{\Az} \fphi d\nuz d\mu,
  }
  where $\Az = A \cap \pi\inv(\zv)$. Therefore, in order to establish the claim, it suffices to show that
  \eq{
    \int_{\Az} \ftphi d\nut_\zv = \int_{\Az} \fphi d\nu_\zv.\nn
  }
  To this end, we have
  \eq{
    \int_{\Az} \ftphi d\nut_\zv
    &\stackrel{\text{(i)}}{=} \int_{\Az} g\qty(\eta\iz\inv(\wv), \zv) d\nuz(\wv) \nn\\
    &= \int_{\eta\iz\inv(\Az)} g\qty\big(\phi(\yv, \zv)) d\nu(\yv)\nn\\
    &= \int_{g\ijz\inv \circ \eta\iz\inv(\Az)} g\qty\big(\phi(\uv, \zv)) d\nu(\uv),\nn
  }
  where (i) follows from the substitution $\yv = \eta\iz\inv(\wv)$, and the last equality follows by taking $\uv = g\ijz\inv(\yv)$ for $g\ijz \in G$ and by noting that $\nu$ is $G$-invariant. Using the ``change of coordinates'' in \eref{eq:charts}, we have $\psi\jz\inv = g\ijz\inv \circ \eta\iz\inv$, and therefore,
  \eq{
    \int_{\Az} \ftphi d\nut_\zv
    &= \int_{\psi\jz\inv(\Az)} g\qty\big(\phi(\uv, \zv)) d\nu(\uv)\nn\\
    &= \int_{\Az} g\qty\big(\phi( \psi\jz\inv(\wv), \zv)) d \qty\big((\psi\jz)_\#\nu)(\wv)\nn\\
    &= \int_{\Az} g\qty\big(\phi( \psi\jz\inv(\wv), \zv)) d \nuz(\wv) = \int_{\Az} \fphi d\nuz, \nn\\
  }
  which proves the claim.\QED

}


{
\subsection*{Proof of Corollary~\ref{cor:discrete}}
\label{proof:cor:discrete}

Let $\X\open = \sqcup_{i=1}^n A_i = \X \setminus \cup_{j=1}^m B_j$. First, note that since $\mu(B_j) = 0$ for each $j \in \pb{1, 2, \dots, m}$, it follows that $\mu\qty(\cup_{j=1}^m B_j) = 0$. Therefore, from Remark~\ref{remark:nonlinear}~\ref{remark:measure-zero}, omitting $\cup_{j=1}^m B_j$ doesn't affect \Fequivalence{}, and it suffices to show that the claim holds for $\X\open$.

To this end, observe that $\X\open$ can be represented as the fiber bundle $\e = \pa(\X\open, \Y, \z, \pi, G)$ where $\z = A$, $\Y = \pb{1, 2, \dots, N}$ and $G$ is isomorphic to $\mathbb{Z}_N$ acting on $\Y$ by addition modulo $N$. Indeed, since $\pi$ is assumed to be a continuously surjective map from $\X$ to $A$, it suffices to show that $\X$ is locally trivializable. Let $V = A$, and for $\pi\inv(A) = \sqcup_{i=1}A_i$ let $\psi: A \times \Y \to \pi\inv(A)$ be given by $\psi(A \times \pb{i}) = \pi\inv(A) \cap A_i = A_i$ for each $i = 1, 2, \dots, N$. Since each $A_i$ is diffeomorphic to $A$, it follows that $\psi$ is a diffeomorphism. Additionally, for $\av \in A$, the map $\psi\inv_{\av}(\xv) = \pb{i: \xv \in A_i}$. Therefore, $\pb{(V, \psi)}$ is a local trivialization, and $\e$ is a fiber bundle. 

Let $\nu$ be the counting measure on $\Y$, and $C = 1/N = 1/\nu(\Y)$. For $\zv = \pi(\xv)$, we may write the density function in \eref{eq:discrete-density}
\eq{
  \fphiv(\xv) = C \cdot g\qty\Big( \phi\qty\big( \psi_{\zv}\inv(\xv), \zv ) ),\nn
}
where the map $\phi(i, \av) = \phi_i(\av)$. Furthermore, the Jacobian constraint~\eref{eq:jac_constraint} requires that
\eq{
  \int_{\Y} C \cdot \Psi_\mu\qty(\absdetj{\phi(i, \cdot)\inv}) d\nu(i) = \sum_{i=1}^N \frac{1}{N} \cdot \Psi_\mu\qty(\absdetj{\phi_i\inv}) = 1.\nn
}
Therefore, from Theorem~\ref{thm:nonlinear_invariance}, it follows that $\pb{\fphiv: \phiv \in \Phiv}$ admits \Fequivalence{}.
\QED

}


{
\subsection{Proof of Theorem~\ref{thm:orthogonal}}
\label{proof:thm:orthogonal}

For each ${\thetav} \in \Theta$, let $\Xv_{\thetav}$ be a random variable with density $\fthetav$. From Lemma~\ref{lem:inclusion}, we know that if the distribution of $Z_{\thetav} \defeq \fthetav(\Xv_{\thetav})$ does not depend on~${\thetav}$, \Fequivalence{} follows. The characteristic function for $Z_{\thetav}$ is given by
\eq{
  \varphi_{\thetav}(t) = \E_{Z_{\thetav}}\pa{e^{itZ_{\thetav}}} = \E_{\xtheta}\pa{e^{it \fthetav(\xtheta)}} = \int_{\X}{e^{it \fthetav(\xv)} \fthetav(\xv) d\xv},\nonumber
}
for $t \in \R$. Using Euler's formula, we can write
\eq{
  \varphi_{\thetav}(t) = \int_{\X}{{\cos(t \fthetav(\xv))} \fthetav(\xv) d\xv} \ \ + \ i \cdot \int_{\X}{{\sin(t \fthetav(\xv))} \fthetav(\xv) d\xv}.\nn
}
Using the Taylor series representation for $\cos\pa{t \fthetav(\xv)}$ and $\sin\pa{t \fthetav(\xv)}$ we get,
\eq{
  \varphi_{\thetav}(t) 
  &= 
  \int_{\X}{\pa{\sum_{k=0}^\infty{\f{(-1)^k}{(2k)!}t^{2k}\fthetav^{2k}(\xv)}}\fthetav(\xv) d\xv} + 
  i \cdot \int_{\X}{\pa{\sum_{k=0}^\infty{\f{(-1)^k}{(2k+1)!}t^{2k+1}\fthetav^{2k+1}(\xv)}}\fthetav(\xv) d\xv}\nn\\
  &\stackrel{\text{(i)}}{=} 
  \sum_{k=0}^\infty \f{(-1)^k t^{2k}}{(2k)!} \int_{\X}{{{\fthetav^{2k}(\xv)}}\fthetav(\xv) d\xv} + 
  i \cdot \sum_{k=0}^\infty \f{(-1)^k t^{2k+1}}{(2k+1)!} \int_{\X}{{\sum_{k=0}^\infty{\fthetav^{2k+1}(\xv)}}\fthetav(\xv) d\xv}\nn\\ 
  &=
  \sum_{k=0}^\infty \f{(-1)^k t^{2k}}{(2k)!} h_{2k}(\thetav) + 
  i \cdot \sum_{k=0}^\infty \f{(-1)^k t^{2k+1}}{(2k+1)!} h_{2k+1}(\thetav)\nn,
}
where $h_k({\thetav}) \defeq \int_{\X}{\fthetav^{k+1}(\xv)d\xv}$ and (i) follows from Fubini's theorem. It follows from this that the characteristic function $\varphi_{\thetav}$ does not depend on ${\thetav}$ if and only if $\textup{Re}\qty(\varphi_{\thetav}(t))$ and $\textup{Im}\qty(\varphi_{\thetav}(t))$ do not depend on $\thetav$ for all $t \in \R$. Therefore, $\varphi_{\thetav}$ does not depend on ${\thetav}$ if and only if the function $h_k({\thetav})$ does not depend on ${\thetav}$ for each $k \in \mathbb N_0$. Equivalently, for each $k \in \mathbb N_0$ we must have that $\f{\partial}{\partial \theta_i} h_k({\thetav}) = 0$ for all $1\le i\le p$, {i.e.,}
$$
  \f{\partial}{\partial \theta_i} h_k({\thetav}) = \f{\partial}{\partial \theta_i} \pa{ \int_{\X}{\fthetav^{k+1}(\xv)d\xv}} = 0.\nn
$$

Under the additional stochastic regularity conditions, using the Lebesgue-dominated convergence theorem we have
\eq{
  \f{\partial}{\partial \theta_i} \pa{ \int_{\X}{\fthetav^{k+1}(\xv)d\xv}} &=  \int_{\X} \f{\partial}{\partial \theta_i}{\fthetav^{k+1}(\xv)d\xv} = \Big< \fthetav^k, \f{\partial}{\partial \theta_i}\fthetav \Big> _{L^2(\X)} = 0,\nn
}
from which the second claim follows.\QED
}
\section{Discussion}
\label{sec:discussion}

In this work, we have studied the framework of topological inference through the lens of classical statistical theory. In the parametric setup, we have investigated cases when the parameters of the statistical model are not sufficient for statistical inference based on their asymptotic limit in the thermodynamic regime. In our case, this is analogous to the property of \betaequivalence{}. We have characterized several conditions under which a parametric family of distributions admits \Fequivalence{}, which also guarantees \betaequivalence{}. When the distributions share an algebraic structure, we are able to describe necessary and sufficient conditions under which this asymptotic identifiability fails. In the absence of the underlying algebraic structure, we have shown that when the distributions satisfy a certain Jacobian constraint, they admit \Fequivalence{}. Lastly, in the absence of any of the above, when the distributions are stochastically regular (as is most often the case), we have shown that if the density function shares a certain geometry with its gradient, then \Fequivalence{} follows. 

As noted in Remark~\ref{remark:betaequivalence}, studying injectivity for Betti numbers collectively serves as a stepping-stone to understanding the behavior of more complex topological invariants, and we have focused on the phenomenon of \Fequivalence{} in the thermodynamic regime. Analogous asymptotic behavior for Betti numbers in the sparse regime has been the focus in \cite{kahle2011random,bobrowski2015topology,Yogeshwaran_2015}. For fixed $1 \le k \le d-1$, in the regime that $nr^d_n \rightarrow 0$ and $n^{(k+2)}r_n^{d(k+1)} \rightarrow \infty$, the following limit theorem is established for random \cech{} complexes in \mbox{\cite[Theorem 3.2]{kahle2011random}} and \cite[Theorem 4.1]{bobrowski2015topology}:
\eq{
  \frac{\beta_k(\K(\Xn, r_n))}{\sqrt{n^{(k+2)}r_n^{d(k+1)}}} \stackrel{d}{\rightsquigarrow} \N(\overline\mu_k, \overline\sigma_k^2),
\label{eq:sparse-limit}
}
where, for fixed quantities $\zeta_k$ and $\xi_k$, the functionals $\mu_k$ and $\sigma_k^2$ are given by
\eq{
  \overline\mu_k = \frac{1}{(k+2)!}\int_\X \zeta_k f(\xv)^{k+2}d\xv, \quad and \quad \overline\sigma^2_k = \frac{1}{(k+1)!}\int_\X \xi_k f(\xv)^{k+1}d\xv. \nn
}
From Theorem~\ref{thm:orthogonal}, we can see that if a family of distributions admit \betaequivalence{}, then for each $1 \le k \le d-1$, in the regime that $nr^d_n \rightarrow 0$ and $n^{(k+2)}r_n^{d(k+1)} \rightarrow \infty$, the limit on the r.h.s.~of Eq.~\eqref{eq:sparse-limit} is identical. A similar conclusion for the Betti numbers of random Rips complexes also follows from \cite[Theorem 3.1]{kahle2011random} in the regime that $n^{(2k+2)}r_n^{d(2k+1)} \rightarrow \infty$. As noted in \cite[page 344]{bobrowski2014topology} this difference in the regimes stems from the fact that the smallest nontrivial $H_k$ cycle in the Rips complex is supported on $2k+2$ vertices as opposed to the $k+2$ vertices required for the \cech{} complex. The results in \cite{bobrowski2022homological} suggest a useful direction for pursuing this line of investigation and is left for future work.

Lastly, it is important to note that the conditions characterized in this work are purely statistical in nature, and hold in the asympotic setting. However, there is still hope that topological summaries for statistical inference from \Fequivalent{} families of distributions may be useful in finite samples and is a promising direction for future work. We hope that this work will serve as a stepping-stone for further investigations in this direction.
\endgroup
\appendix



\numberwithin{equation}{section}
\section{Supplementary Results}
\label{supplementary}

\begin{lemma}
  Suppose $f$ is a probability density function with mean $0$ and variance $1$. For each $d \in \mathbb{Z}_{+}$, let $\xv \in \R^d$ and $\zeta_d : \R_+ \rightarrow \R_+$. Then $\prod_{i=1}^d f(x_i) = \zeta_d(\norm{\xv}^2)$
  holds for each $d \in \mathbb N_0$ if and only if $f(x) = \exp\qty\big({-{x^2}/{2}})/{\sqrt{2\pi}}$, for each $x \in \R$.
  \label{lemma:normal_density}
\end{lemma}
\prf{
The sufficient condition follows unambiguously by plugging in the value for $f(x_i)$, i.e., when $f(x) = \exp\qty\big({-{x^2}/{2}})/{\sqrt{2\pi}}$ it follows that
\eq{
\prod_{i=1}^d f(x_i) = \pa{{2\pi}}^{-{d}/{2}} \exp\pa{-\norm{\xv}^2 / 2} \defeq \zeta_d(\norm{\xv}^2).\nn
}
It remains to verify the necessary condition. To begin, consider the case when $d=1$. For $x \in \R$, we have $f(x) = \zeta_1(x^2)$. Define $k\defeq f(0) = \zeta_1(0)$. We now proceed to consider the case when $d=2$. For $(x,0) \in \R^2$, we have $f(x) \cdot f(0) = \zeta_2(x^2)$. Since $f(x) = \zeta_1(x^2)$ and $f(0)=k$, it follows that $\zeta_2(x^2) = k\zeta_1(x^2)$ for all $x \in \R$. By induction, for each $d \in \mathbb{Z}_+$, we have
\eq{
  \zeta_d(x^2) = k^{d-1} \zeta_1(x^2).
  \label{eq:zeta1}
}
{Thus, for any $\xv \in \R^d$}, Eq.~\eqref{eq:zeta1} implies
\eq{
  \prod_{i=1}^d f(x_i) = \zeta_d(x_1^2 + x_2^2 + \dots + x_d^2) = k^{d-1} \zeta_1(x_1^2 + x_2^2 + \dots + x_d^2).\nonumber
}
We also have that $\prod_{i=1}^d f(x_i) = \prod_{i=1}^d \zeta_1(x_i^2)$. Define $g(x) = \f{\zeta_1(x)}{k}$, which implies $g$ satisfies
\begin{equation}
  \label{eq:zeta3}
  \prod_{i=1}^d g(x_i^2) = g(x_1^2 + x_2^2 + \dots + x_d^2).
\end{equation}
Eq.~(\ref{eq:zeta3}) holds if and only if $g(x^2) = e^{\beta x^2}$, for some fixed $\beta \in \R$, which implies
\eq{
  f(x) = \zeta_1(x^2) = k\cdot g(x^2) = k\cdot e^{\beta x^2}.\nn
}
However, $f$ is a probability density function on $\R$ with mean $0$ and variance $1$, i.e, $\int_{\R}{f(x)dx} = 1$, $\int_{\R}{xf(x)dx} = 0$ and $\int_{\R}{x^2f(x)dx} = 1$.
This yields $\beta = -\f{1}{2}$, $k = \f{1}{\sqrt{2\pi}}$ and the result follows.
}


{
  \forcecommand{\pv}{\boldsymbol{p}}
  \begin{example}[Example~\ref{ex:mobius} Continued]
    \label{ex:mobius-appendix}

    \begin{figure}
      \centering
      \begin{subfigure}[b]{\linewidth}
        \includegraphics[width=0.5\linewidth]{\Root/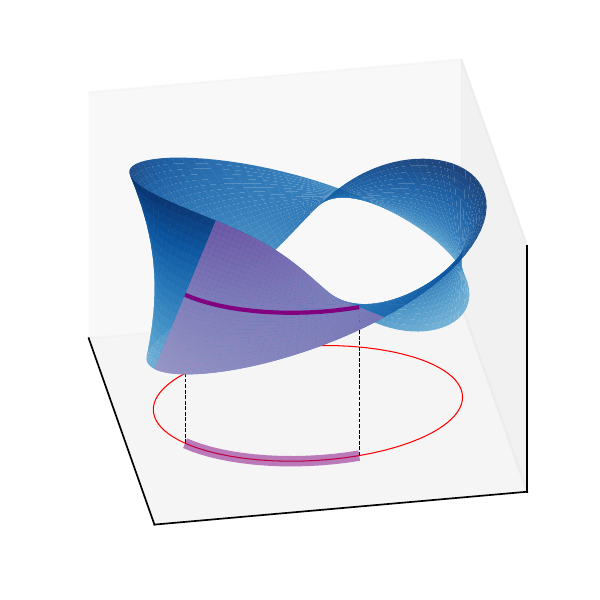}
        \includegraphics[width=0.5\linewidth]{\Root/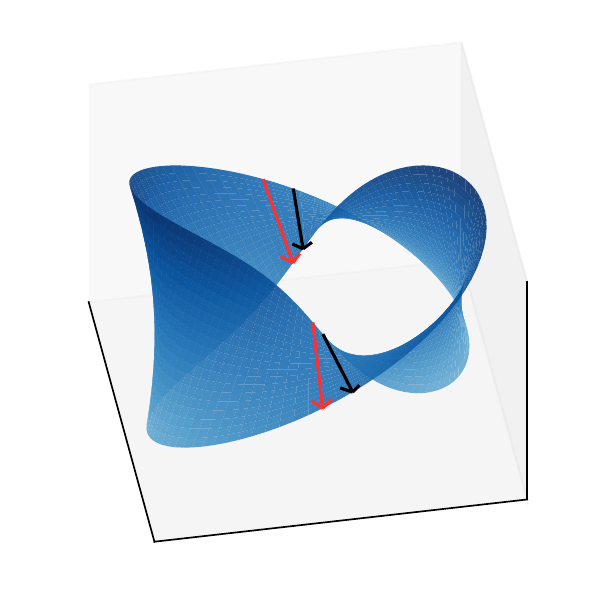}
        \caption{An illustration of the space $X$---the Möbius band with the full twist. (Left) The projection map $\pi$ onto a neighborhood $U \subset [0, 2\pi]/\sim = \z$. (Right) The map $\psi_\theta(t)$ for $t \in [-0.5, 0.5]$ preserves length and orientation, unlike the Möbius band with a half-twist. Therefore, the measure induced on the fiber over $\theta$, $\pi\inv(\theta)$, is also a unit uniform distribution.}
        \label{fig:mobius-appendix-a}
      \end{subfigure}

      \begin{subfigure}[t]{0.97\linewidth}
        \includegraphics[width=0.32\linewidth]{\Root/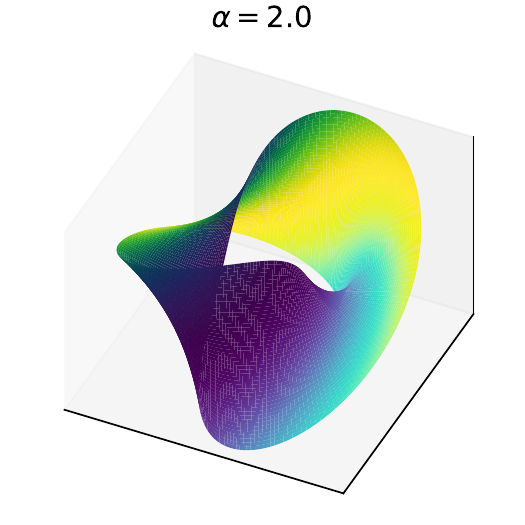}
        \includegraphics[width=0.32\linewidth]{\Root/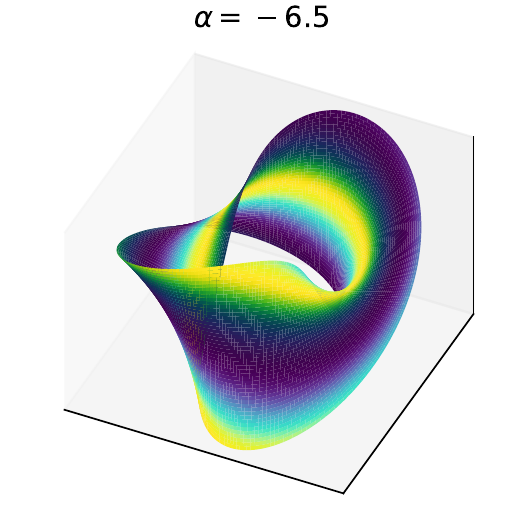}
        \includegraphics[width=0.32\linewidth]{\Root/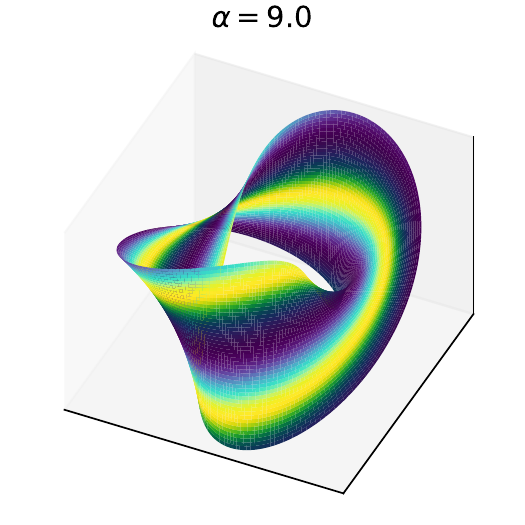}
        \caption{Superlevel sets of the probability density function $\fa$ on $\X$ for $\alpha \in {2.0, -6.5, 9.0}$.}
        \label{fig:mobius-appendix-b}
      \end{subfigure}
      \hspace*{-1em}
      \begin{subfigure}[t]{0.01\linewidth}
        \includegraphics[width=4\linewidth]{\Root/inputs/plots/mobius_cb.pdf}
      \end{subfigure}

      \caption{Illustration of the family of distributions from Example~\ref{ex:mobius}.}
      \label{fig:mobius-appendix}
    \end{figure}

    The space $\X$ is the surface of a Möbius band with a full twist, $\X$, given by
    \eq{
    \X = \qty\Big{ \xv(\theta, t) \in \R^3: \theta \in [0, 2\pi), \ t \in [-0.5, 0.5]},\nn
    }
    where
    \eq{
      \xv(\theta, t) = \qty\Big(\cos\theta(1 + t\cos\theta), \  \sin\theta(1 + t \cos\theta), \ t\sin\theta).\nn
    }
    This can be represented as a fiber bundle $\e = (\X, \z, \Y, \pi, G)$, with the base space $\z = \mathcal{S}^1$, the typical fiber $\Y = [-0.5, 0.5]$, and the projection map given by $\pi(\xv) = \arctan({x_2/x_1})$ by identifying $\z \simeq [0, 2\pi]/\sim$ where $\pb{0} \sim \pb{2\pi}$.

    Consider the sets $V_1 = (0, 2\pi) \setminus \pb{\pi}$, and $V_2 = [0, \epsilon) \cup (\pi - \epsilon, \pi + \epsilon)$ for $0 < \epsilon < \pi/2$, and for $\theta = \pi(\xv) \in [0, 2\pi)$
    \eq{
      \psi\inv_{1, \theta}(\xv) = \frac{x_3}{\sin \theta}, \qq{and} \psi\inv_{2, \theta}(\xv) = \frac{1}{\cos \theta}\qty(\frac{x_1}{\cos \theta} - 1),\nn
    }
    The collection $\pb{(V_j, \psi_j): j = 1, 2}$ provides a local trivialization for $\X$. Additionally, the structure group $G = \pb{\id}$ is simply the identity element since $\psi\inv_{i, \theta} \circ \psi_{j, \theta} = \id$ for all $\theta \in V_i \cap V_j,$ and  $i, j  \in \qty{1, 2}$. It is straightforward to verify that for $\theta \in V_j \subset [0, 2\pi)$, $t \in [-0.5, 0.5]$, and $\xv = \xv(\theta, t)$, it follows that $\pi(\xv) = \theta$ and $\psi\inv_{j, \theta}(\xv) = t$. 

    Let $\nu \sim \textup{Unif}([-0.5, 0.5])$ be the uniform distribution on $\Y = [-0.5, 0.5]$ with $d\nu(t) = 1 \cdot dt$ for all $t \in [-0.5, 0.5]$. For a fixed $\theta \in V_j$, the map $\psi_{j, \theta}$ pushes forward the measure $\nu_\theta$ to the fiber $\pi\inv(\theta)$. The image $\psi_{j, \theta}\qty(\qty[-0.5, 0.5])$ is a line segment in $\R^3$ of length 
    $$
      \norm{\psi_{j, \theta}(0.5) - \psi_{j, \theta}(-0.5)} = \norm{ (\cos^2\theta, \sin\theta\cos\theta, \sin\theta) } = 1,
    $$
    as shown in Figure~\ref{fig:mobius-appendix-a}. Therefore, the pushforward measure, $\nu_\theta = (\psi_{j, \theta})_\# \nu$, is also a uniform distribution with $d\nu_\theta(\xv) = 1$ for all $\xv \in \pi\inv(\theta)$. Hence, the density $\fa$ in \eref{eq:fa-mobius} can be faithfully represented as
    \eq{
      \fa(\xv) = g\qty( \phi_\alpha(\psi\inv_{j, \theta}(\xv), \theta) ) d\nu_\theta(\xv).\nn
    }
    The \Fequivalence{} of $\pb{\fa: \alpha \in \R}$ now follows from Example~\ref{ex:mobius}.
  \end{example}

  \begin{figure}
    \centering
    \begin{subfigure}[c]{0.95\linewidth}
      \includegraphics[width=0.48\linewidth]{\Root/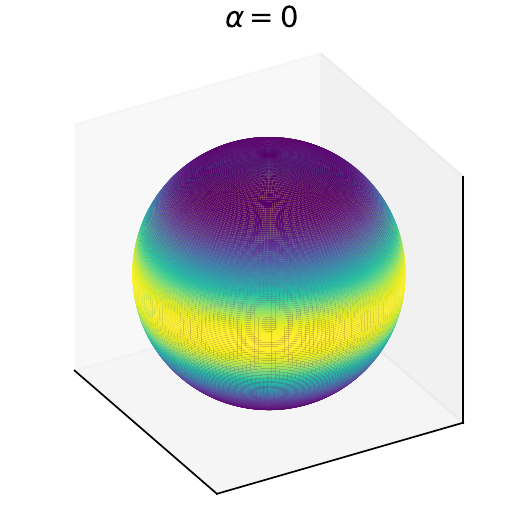}
      \includegraphics[width=0.48\linewidth]{\Root/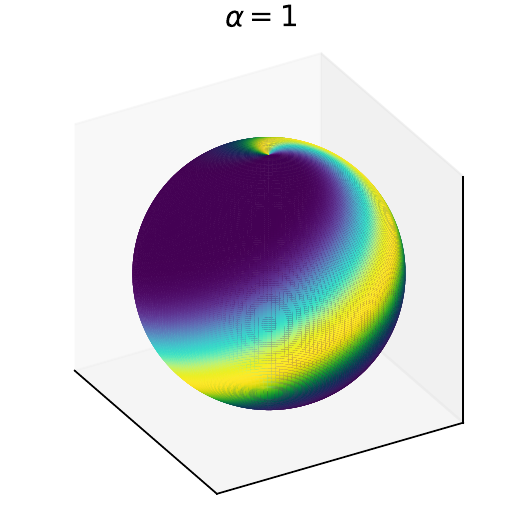}

      \includegraphics[width=0.48\linewidth]{\Root/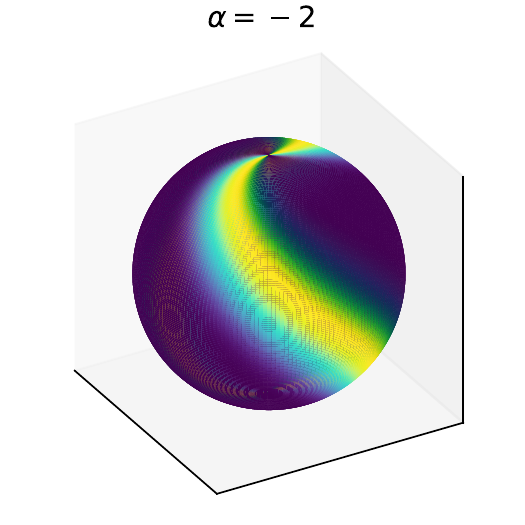}
      \includegraphics[width=0.48\linewidth]{\Root/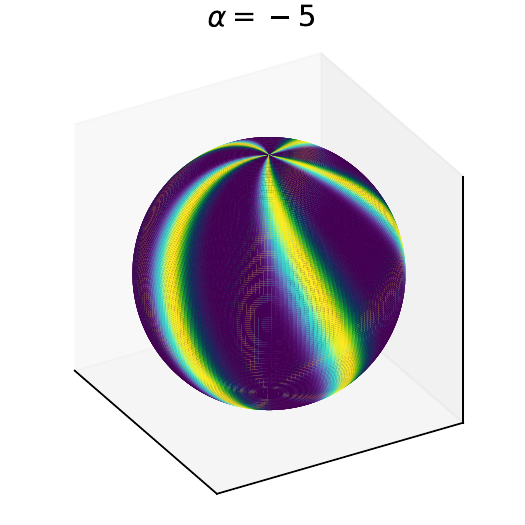}
    \end{subfigure}
    \hspace*{-3em}
    \begin{subfigure}[c]{0.03\linewidth}
      \includegraphics[height=20\linewidth]{\Root/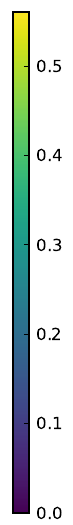}
    \end{subfigure}
    \caption{Superlevel sets of the density function $\fa$ on $\X=\S^2 \setminus \pb{\pv, -\pv}$ from Example~\ref{ex:sphere} when $g \sim \textup{Beta}(10, 10)$ is the density function of a Beta distribution on $(0, 1)$, and for $\alpha\in \pb{0, 1, -2, -5}$}
    \label{fig:sphere}
  \end{figure}

  \begin{example}[Family of distributions on $\S^2$]
    \label{ex:sphere}
    
    Let $\X = \S^2 \setminus \pb{+\pv, -\pv}$ be the surface of the unit sphere embedded in $\R^3$, excluding the two poles. This can be represented in polar coordinates as
    \eq{
    \X = \qty\Big{ \xv(\theta, \varphi) = \qty(\sin \theta \cos \varphi, \sin \theta \sin \varphi, \cos \theta) \in \R^3 : \theta \in (0, \pi), \ \ {and}\ \ \varphi \in [0, 2\pi) }.\nn
    }
    Let $\fa$ be a probability density function on $\X$ given by
    \eq{\label{eq:fa-sphere}
      \fa(\xv) = \fa\qty\big(\xv(\varphi, \theta)) = \frac{1}{2\pi} \cdot g\qty(\frac{\theta + \alpha \varphi/2 \mod \pi}{\pi}),
    }
    where $g$ is a fixed probability density function with $\supp(g) = (0, 1)$ w.r.t. the Lebesgue measure. The family $\pb{\fa: \alpha \in \Z}$ admits \Fequivalence{}. 
    
    This is similar to Example~\ref{ex:mobius}, and can be understood by considering the fiber bundle representation of $\X$ with $\z = (0, \pi)$, $\Y = \S^1 \simeq [0, 2\pi]/\sim$ and the projection map $\pi(\xv) = \arctan(x_2/x_1)$. The local trivialization is given by $\pb{(V, \psi)}$, where 
    \eq{
      V = (0, \pi), \ \psi\inv(\xv) = \qty\Big(\arctan\qty({x_2}/{x_1}), \arccos\qty(x_3)), \text{ and } \psi\inv_\theta(\xv) = \arccos(x_3).\nonumber 
    }
    The uniform measure on $\Y$ is $d\nu(\varphi) = 1/2\pi$, and for $\xv \in \X$ with $\theta = \pi(\xv)$ the density in \eref{eq:fa-sphere} can be rewritten as 
    \eq{
      \fa(\xv) = C \cdot g\qty( \phi\qty(\psi\inv_{\theta}(\xv), \theta) ),\nn
    }
    where $C = 1/2\pi$, and the map $\phi(\varphi, \theta) = \theta + \alpha \varphi / 2$ satisfies $\phi(\varphi, \cdot) \in \Delta\qty((0, \pi))$ for all $\varphi \in [0, 2\pi)$. Since $\S^1 \simeq [0, 2\pi]/\sim$ requires $0 \sim 2\pi$, we have 
    \eq{
      \phi(0, \theta) = \phi(2\pi, \theta) \quad \Longleftrightarrow \quad \pi\alpha \mod \pi = 0, \nn
    }
    which holds whenever $\alpha \in \Z$.

    As $g$ is a density with respect to the Lebesgue measure, the modular character is the identity map, and the Jacobian constraint in \eref{eq:jac_constraint}, 
    \eq{ 
      \int_{0}^{2\pi} C \cdot \absdetj{\phi\inv(\varphi, \cdot)} d\nu(\varphi) = \int_{0}^{2\pi} \abs{\frac{d}{d\theta}(\theta - \alpha \varphi/2 \mod \pi)} \frac{1}{2\pi}d\varphi = \int_{0}^{2\pi} 1 \cdot \frac{1}{2\pi}d\varphi = 1, \nn
    } 
    is satisfied for all $\alpha \in \Z$. From Theorem~\ref{thm:nonlinear_invariance}, we conclude that the family of distributions $\pb{\fa: \alpha \in \Z}$ admits \Fequivalence{}. 

    Furthermore, according to Remark~\ref{remark:nonlinear}\ref{remark:measure-zero}, this family of distributions can be continuously extended to a family of distributions supported on $\S^2$. Figure~\ref{fig:sphere} illustrates the superlevel sets $\pb{\xv \in \X: \fa(\xv) \ge t}$ for four different values of $\alpha$ when $g \sim \textup{Beta}(10, 10)$ is the density function of a Beta distribution on $(0, 1)$.
  \end{example}


}%
\bibliographystyle{plainnat}
\bibliography{\Root/main}
\endgroup
\end{document}